\documentclass[12pt,a4paper]{article}
\usepackage[T2A]{fontenc}
\usepackage[cp1251]{inputenc}
\usepackage{amssymb,epic,eepic,color, graphicx}
\usepackage{fancyhdr}
\usepackage{amsthm}
\usepackage{amsmath}
\usepackage{setspace}
\usepackage{fancyhdr}
\newtheorem{theorem}{Theorem}

\newtheorem{lemma}{Lemma}[section]
\newtheorem{proposition}{Proposition}[section]

\begin{document}

{\bf On the topology of 3-manifolds admitting Morse-Smale diffeomorphisms with four fixed points of pairwise different Morse indices }

O.Pochinka, E. Talanova

\begin{abstract} In the present paper we consider class  $G$ of orientation preserving Morse-Smale diffeomorphisms $f$, which defined on closed 3-manifold $M^3$, and whose non-wandering set consist of four fixed points with pairwise different Morse indices. It follows from S. Smale and K. Meyer results that all gradient-like flows with similar properties has Morse energy function with four critical points of pairwise different Morse indices. This implies, that supporting manifold $M^3$ for these flows admits a Heegaard  decomposition of genus 1 and hence it is homeomorphic to a lens space $L_{p,q}$. Despite the simple structure of the non-wandering set in class $G$ there exist diffeomorphisms  with wild embedded  separatrices. According to  V. Grines, F. Laudenbach, O. Pochinka results such diffeomorphisms do not possesses an energy function, and question about topology their supporting manifold is open. According to  V. Grines, E. Zhuzhoma and V. Medvedev  results  $M^3$ is homeomorphic to a lens space $L_{p,q}$ in case of tame embedding of closures of one-dimensional  separatrices of diffeomorphism $f\in G$. Moreover, the wandering set of $f$ contains at least $p$ non-compact heteroclinic curves.
In the present paper similar result was received for arbitrary diffeomorphisms of class $G$.
Also we construct diffeomorphisms from $G$ with wild embedding one-dimensional  separatrices on every lens space $L_{p,q}$.  Such examples were known previously only on the 3-sphere.
\end{abstract}

\section{Formulation of results}
It's well known that the Morse-Smale systems exist on any manifolds. These systems describe regular (non-chaotic) processes in technology. They have finite hyperbolic non-wandering set, which is fully described by numbers orbits of different {\it Morse indices} (the dimension of their unstable manifold). A natural question arises, what we say about the supporting manifold topology of such system, if we know the structure of it non-wandering set. The classic example of an exhaustive answer to the question are the systems with two points of extremal Morse indices.
In this case, it follows from Reeb's theorem \cite{Reeb}, that the supporting manifold is homeomorphic to the $n$-sphere. Another example of following global properties from local ones is the equality for a gradient-like system of the alternating sum of number periodic points of different Morse indices to Euler characteristic of the ambient surface. For flows this fact  follows from classical Poincare-Hopf theorem \cite{Poincare}, \cite{Hopf} and for cascades it follows from the existence of a  Morse energy function, proved by D. Pixton \cite{Pixton1977}, and from Morse inequality.

For the dimension equals 3 this  equality also is true. However, the Euler characteristic of all closed orientable 3-manifolds is equal to zero and, accordingly, it does not shed any light on the  supporting space topology. A more cunning play with the numbers of periodic points of different Morse indices leads to the fact that for flows it is possible to find a connection between these numbers and the genus of the Heegaard decomposition of a 3-manifold. In the absence of a topological classification of 3-manifolds, an information about the Heegaard decomposition of a given manifold is very informative and identifying in some cases. A similar question for diffeomorphisms is open today due to the possibility of wild behaviour of the saddle point separatrices, first discovered by D. Pixton \cite{Pixton1977}.

The effect of the possibly wild embedding of saddle separatrices of a Morse-Smale 3-diffeomorphism into an ambient manifold had a revolutionary impact on the understanding of the dynamics of such systems. It became clear that their description does not fit into the framework of purely combinatorial invariants, and requires the involvement of a topological apparatus. Nevertheless, a complete topological classification of Morse-Smale 3-diffeomorphisms, including the realization, was obtained in the works of C. Bonatti, V. Grines and O. Pochinka \cite{BoGrPo2019}, \cite{BoGrPo2017}. However, these invariants does not answer the question, whether this 3-manifold admits a gradient-like diffeomorphism with wildly embedded saddle separatrices or not. 

In papers of V. Grines, E. Zhuzhoma and V. Medvedev  \cite{GrMeZh2003} it was established that gradient-like diffeomorphisms with tame  embedded one-dimensional saddle separatrices are look like to flows, therefore the structure of the non-wandering set of such a diffeomorphism uniquely determines the Heegaard decomposition of its supporting manifold. In some partial cases this result was generalized on mildly wild embedding of separatrices by V. Grines, F. Laudenbach and O. Pochinka \cite{GrLauPo2012}, and on diffeomorphisms with a single saddle point with wild separatrices -- by C. Bonatti and V. Grines \cite{BoGr2000}. In both case the Heegaard's genus of the ambient manifold is 0, that is $M^3$ is the 3-sphere. 

In the present paper we consider the class $G$ of orientation-preserving Morse-Smale diffeomorphisms, defined on a closed 3-manifold $M^3$, whose non-wandering set consists of exactly four points of pairwise different Morse indices. It follows from the results of S. Smale \cite{Smale1961} and K. Meyer \cite{Meyer1968} that all gradient-like flows with similar properties (see Fig. \ref{pic1}) has a Morse energy function with exactly four critical points of pairwise different indices. It immediately implies that the supporting manifold $M^3$ for such flows admits a Heegaard decomposition of genus 1 and, therefore, it is homeomorphic to a lens space $L_{p,q}$ (see, for example, \cite{Fomad}).
\begin{figure}[h]\center{\includegraphics[width=0.5\linewidth]{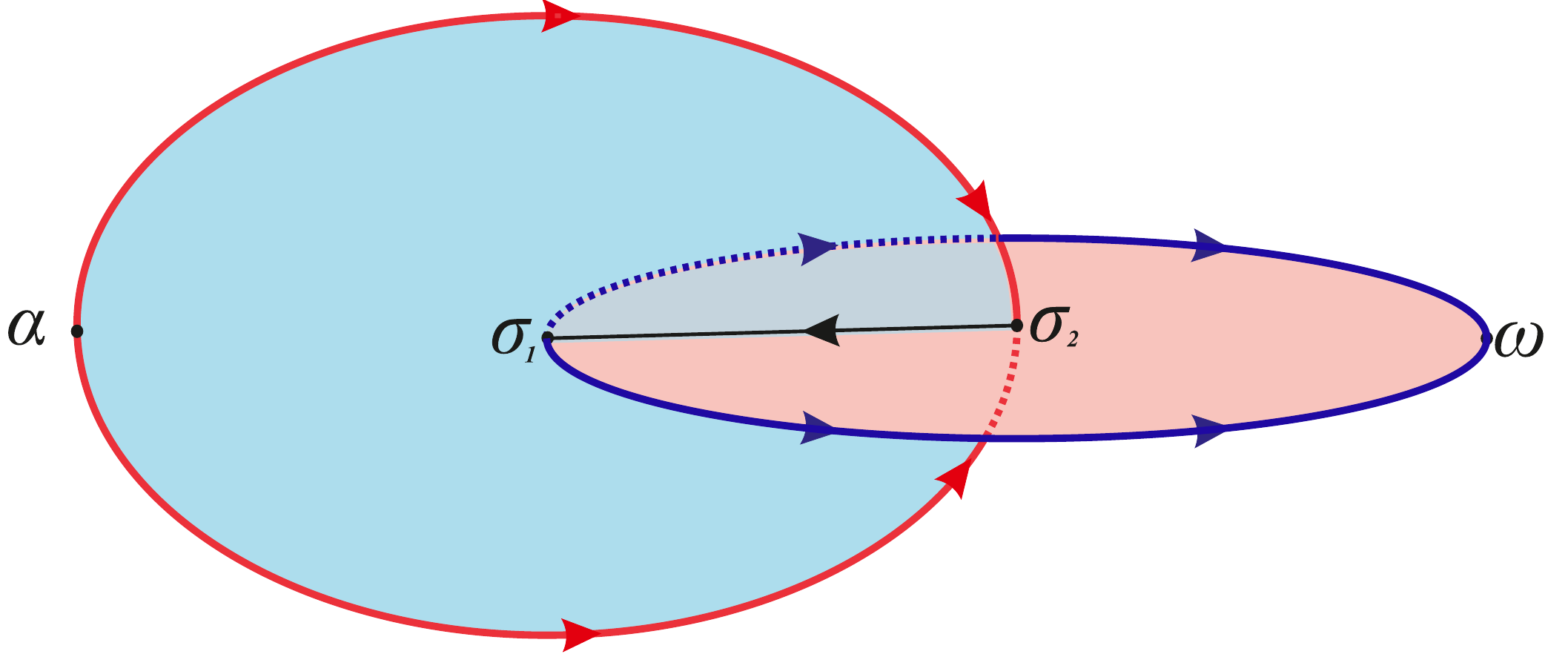}}\caption{Gradient-like 3-flow with four points of pairwise distinct Morse indices on lens $L_{1,0}\cong\mathbb S^3$}\label{pic1}\end{figure}

Despite the simple structure of the non-wandering set, there are diffeomorphisms in the class $G$ with wildly embedded saddle separatrices \cite{PoTaSh} (see Fig.  \ref{af2}). However, all currently known examples were constructed on the 3-sphere. 
\begin{figure}[h]\center{\includegraphics[width=0.5\linewidth]{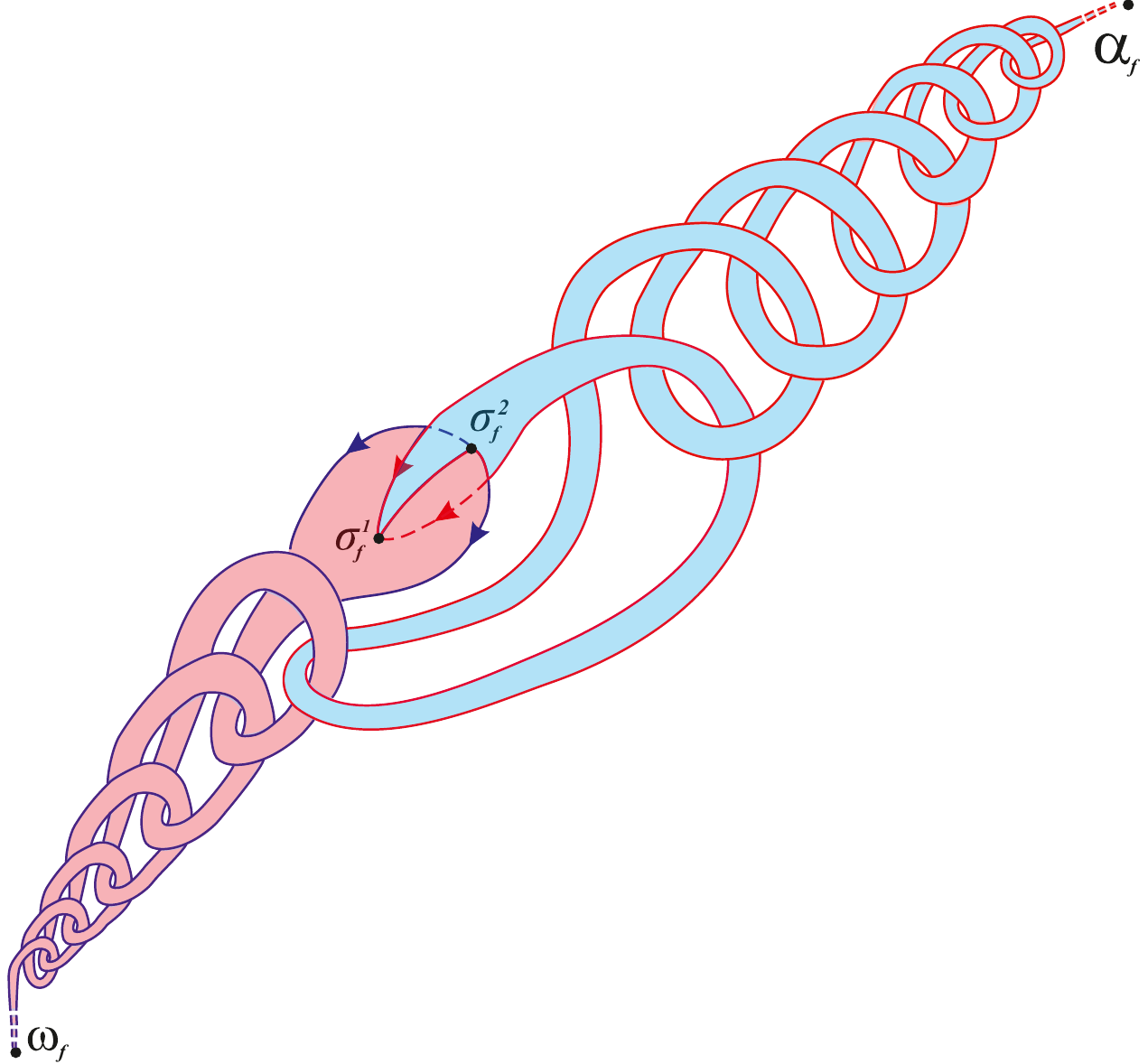}}\caption{Diffeomorphism
from the class $G$ with wildly embedded saddle separatrices}\label{af2}\end{figure}

One of the results of this paper is a constructive proof of the following fact.
\begin{theorem}\label{exi} On any lens space $L_{p,q}$ there exists a diffeomorphism $f\in G$ with wildly embedded one-dimensional saddle separatrices.
\end{theorem}

According to \cite{GrLauPo2012} such diffeomorphisms do not have a Morse energy function, and the question of the topology of their supporting manifold has remained open until today. According to  \cite{GrMeZh2003}, in the case of tame embedding of the closures one-dimensional separatrices of the diffeomorphism $f\in G$, the supporting manifold $M^3$ is homeomorphic to the lens space  $L_{p,q}$. The wandering set of the diffeomorphism $f$ contains at least $p$ non-compact heteroclinic curves. 

In the present paper, a similar result is obtained for arbitrary diffeomorphisms of the class $G$. In more detail.

Let $f\in G$. It follows from the definition of the class that the non-wandering set of $f$ consists of exactly four points $\omega_f,\sigma_f^{1},\sigma_f^{2},\alpha_f$ with Morse indices $0,1,2,3$, respectively. Thus $f$ has exactly two saddle points $\sigma_f^1,\,\sigma_f^2$ of Morse indices 1 and 2, respectively, the intersection of two-dimensional manifolds of which forms a heteroclinic set $$H_f=W^s_{\sigma_f^1}\cap W^u_{\sigma_f^2}.$$
We introduce the concept of the heteroclinic index $I_f$ of $f$ as follows. If the set $H_f$ does not contain non-compact curves, then we assume $I_f=0$. Otherwise, any non-compact curve $\gamma\subset H_f$ contains, together with any point $x\in\gamma$, a point $f(x)$. We will consider the curve $\gamma$ oriented in the direction from $x$ to $f(x)$. We will also fix the orientation on manifolds $W^s_{\sigma_1}$ and $W^u_{\sigma_2}$. For a non-compact heteroclinic curve $\gamma$, we denote by $$v_\gamma=(\vec v^1_\gamma,\vec v^2_\gamma,\vec v^3_\gamma)$$ a triple of vectors with the origin at the point $x\in\gamma$ such that $\vec v^1_\gamma$ -- normal vector to $W^s_{\sigma_1}$, $\vec v^2_\gamma$ -- normal vector to $W^u_{\sigma_2}$ and $\vec v^3_\gamma$ -- tangent vector to the oriented curve $\gamma$. Let's put $I_\gamma=+1\,(I_\gamma=-1)$ in the case of right (left) orientation of $v_\gamma$. The number $$I_{f}=\left|\sum\limits_{\gamma\subset H_f}I_\gamma\right|$$ is called {\it the heteroclinic index of diffeomorphism $f$}. For an integer $p\geqslant 0$, we denote by $G_p\subset G$ a subset of diffeomorphisms $f\in G$ such that $I_{f}=p$ (see Fig. \ref{neor}).
\begin{figure}[h]
\center{\includegraphics[width=0.5\linewidth]{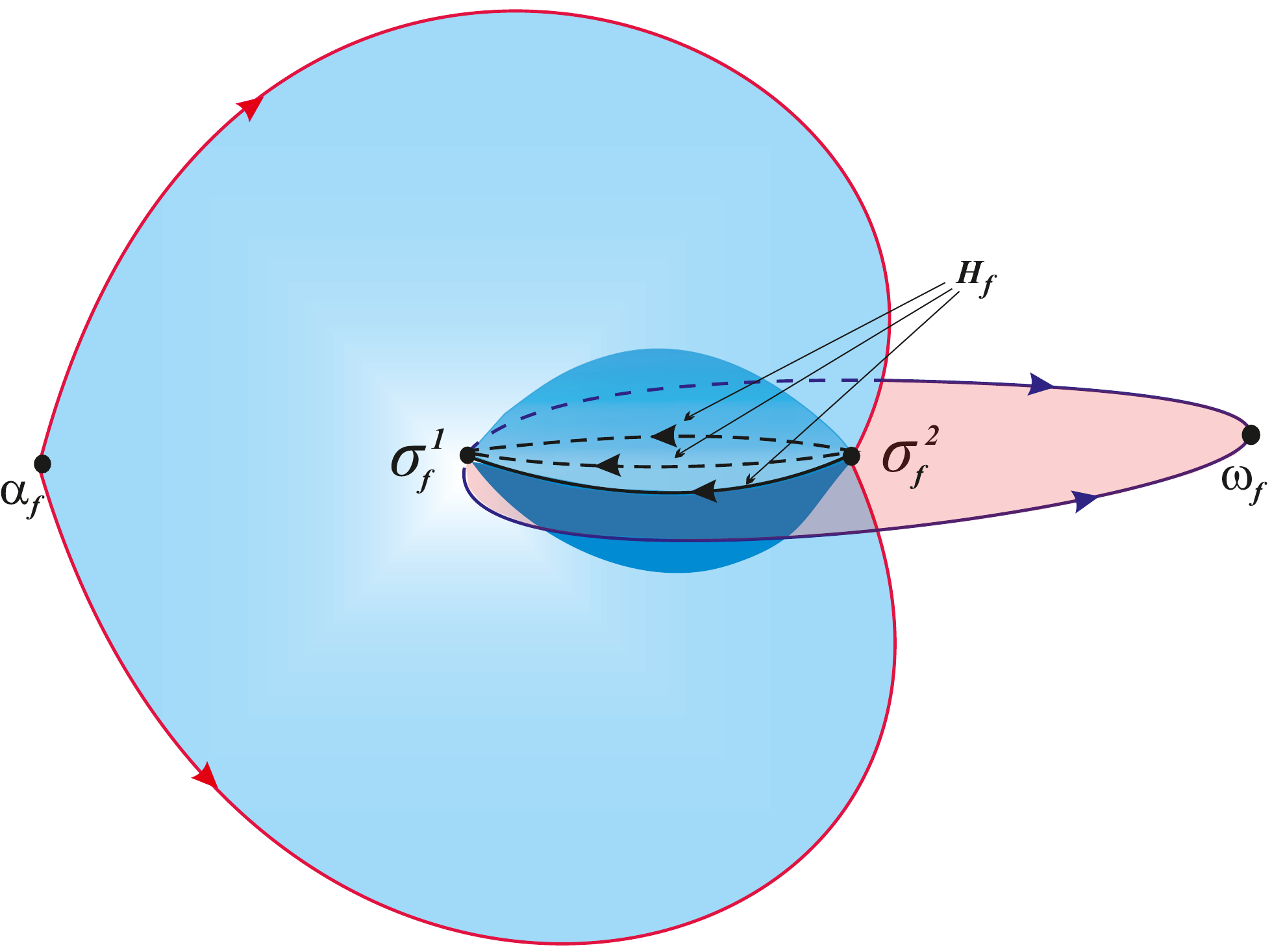}}
\caption{\small Diffeomorphism $f\in G$ with non-orientable set $H_f$ consisting of three non-compact curves and heteroclinic index 1}
\label{neor}
\end{figure}  
The main result of the work is the proof of the following fact.
\begin{theorem}\label{TT1} If a  manifold $M^3$ admits a diffeomorphism ${f}\in{G}_p$ then  $M^3$ is homeomorphic to a lens space $L_{p,q}$.
\end{theorem}  

{\it Acknowledgments.} The study was supported by a grant from the Russian Science Foundation, contract 23-71-30008.

\section{Necessary definitions and facts}
\subsection{Topology}
For any subset $X$ of the topological space $Y$ we will denote by $i_X:X\to Y$ the {\it inclusion map}. For any continuous map $\phi:X\to Y$ from the topological space $X$ to the topological space $Y$ will be denoted by $\phi_*:\pi_1(X)\to\pi_1(Y)$ its {\it induced homomorphism}.

By {\it $C^{r}$-embedding} $(r\geqslant 0)$ of a  manifold $X$ into a manifold $Y$  is called a map  $f:X\to Y$ such that $f:X\to f(X)$ is a $C^r$-diffeomorphism. $C^0$-embedding is also called a {\it topological embedding}.

The topological embedding $\lambda:X\longrightarrow Y$ of an  $m$-manifolds $X$ into an $n$- manifold $Y$ ($n\leqslant m$) is called {\it locally flat at a point} $\lambda(x)$, $x\in X$, if the point $\lambda(x)$ belongs to  a local chart $(U,\psi)$ of the manifold $Y$, that $\psi (U\cap\lambda(X)) =\mathbb{R}^{m}$, where $\mathbb{R}^{m}\subset\mathbb{R}^{n}$ -- the set of points whose last $n-m$ coordinates are 0 or $\psi(U\cap\lambda(X)) = R^{m}_{+}$, where $\mathbb{R}^{m}_{+}\subset\mathbb{R}^{m}$ -- the set of points whose last coordinate is non-negative. The embedding $\lambda$ is called {\it tame}, and a manifold $X$ is {\it tamely embedded}, if $\lambda$ is locally flat at every point $\lambda(x)$, $x\in X$. Otherwise, the embedding  $\lambda$ is called {\it wild}, and the manifold $X$ is {\it wildly embedded}. Any point $\lambda(x)$ that is not locally flat is called a {\it wildness point}.

Let 
$\mathbb{D}^{n} = \{(x_{1},\dots, x_{n})\in \mathbb{R}^{n}:\sum\limits_{i=1}^{n}x^{2}_{i}\leqslant 1\}$ --- {\it is a standard $n$-disk  (ball)}, $\mathbb{D}^{0} = \{0\}$, 
$\mathbb{S}^{n-1} = \{(x_{1},\dots, x_{n})\in \mathbb{R}^{n}:\sum\limits_{i = 1}^{n}x^{2}_{i} = 1\}$ --- {\ it is a standard $(n-1)$-sphere}, $\mathbb{S}^{-1} = \emptyset$. 

\begin{proposition}[\cite{BoGrMePe2002}, Lemma 2.1]\label{sphere-good} Let $\lambda\colon\mathbb{S}^2
\to M^3$ be a topological embedding that is smooth everywhere except at one point $s_{0}$, $x _{0}= \lambda(s_{0})$, $\Sigma=\lambda(\mathbb{S}^2) $, $y_0\in\Sigma\setminus\{x_0\} $ be a fixed point and $V$ be a fixed neighborhood of the sphere $\Sigma$. Then there exists a smooth 3-ball $B$ contained in $V$ such that $x_0\in B$ and $\partial B$ transversally intersects $\Sigma$ along a single curve separating the points $x_0$ and $y_0$ in $\Sigma$.\end{proposition}

A topologically embedded into $n$-manifold $X$ $(n-1)$-sphere ${S}^{n-1}$ is called {\it cylindrical or cylindrically embedded} if there is a topological embedding $e:\mathbb{S}^{n-1}\times[-1,1]\to X$ such that $e(\mathbb{S}^{n-1}\times\{0\}) ={S}^{n-1}$.

An $n$-manifold $X$ is called {\it irreducible} if any $(n-1)$-sphere cylindrical embedded in $X$ bounds an $n$-ball there.

A 3-manifold $X$ is called {\it simple} if it is either irreducible or homeomorphic to $\mathbb{S}^{2}\times{\mathbb{S}^{1}}.$

A surface $F$ topologically embedded into a $3$-manifold $X$ is called {\it properly embedded} if $\partial X\cap F = \partial F$. A properly  embedded into $X$  a surface $F$ is called {\it compressible} in $X$ in one of the following two cases: 

1) there is a non-contractible  simple closed curve $c\subset int F$ and an embedded 2-disk $D\subset int X$ such that $D\cap F = \partial D = c$; 

2) there is a 3-ball $B\subset int X$ such that $F = \partial B$. 

A surface $F$ is called {\it incompressible} in $X$ if it is not compressible in $X$. 

\begin{proposition}[\cite{BoGr2000}, Theorem 4]\label{f_1} Let $T$ be a two-dimensional torus smoothly embedded in the manifold $\mathbb{S}^{2}\times\mathbb{S}^{1}$ so that $i_{T*}(\pi_1(T))\neq 0$. Then $T$ is a boundary of a solid torus, smoothly embedded into $\mathbb{S}^{2}\times\mathbb{S}^{1}$.
\end{proposition}

\begin{proposition}[\cite{BoGr2000}, Lemma 3.1]\label{ss2s1} Let $S$ be a two-dimensional sphere cylindrical embedded in the manifold $\mathbb{S}^{2}\times\mathbb{S}^{1}$. Then $S$ either bounds a 3-ball there, or is ambiently isotopic to the sphere $\mathbb S^2\times\{s_0\},\,s_0\in\mathbb S^1$.
\end{proposition}

\begin{proposition}[\cite{Neu}, Exercise 6]\label{ex_6} Any two-sided compressible 2-torus $T$
in an irreducible 3-manifold $X$
either restricts the solid torus, or it is contained in a 3-ball  there. 
\end{proposition}

\begin{proposition}[\cite{HuWa}, Chapter 4, sec. 5, corollary 1]\label{sl_1} Any $n$-dimensional manifold cannot be separated by a subset of dimension $\leqslant n -2$. 
\end{proposition}

\subsection{Morse-Smale diffeomorphisms}
Let $M^n$ be a smooth closed $n$-manifold with the metric $d$ and $f:M^n\to M^n$ be an orientation-preserving diffeomorphism.

A compact $f$-invariant set $A\subset M^n$ is called an {\it attractor} of a diffeomorphism $f$ if it has a compact neighborhood $U_{A}$ such that $f(U_{A})\subset int(U_{A})$ and $A = \underset{k\geqslant 0}{\bigcap} f^{k}(U_{A})$. The neighborhood $U_{A}$ is called {\it trapping}.
{\it Repeller} is defined as an attractor for $f^{-1}$.

A point $x\in M^n$ is called a {\it wandering point} of the diffeomorphism $f$ if it has a neighborhood $U_x\subset M^n$ such that $f^k(U_x)\cap U_x=\emptyset$ for any $k\neq 0$. The complement to the set of wandering points is called a {\it non-wandering set} of the diffeomorphism $f$.

If the non-wandering set of $f$ is finite that it consists of periodic points. An isolated periodic point $p$ of the period $m_p$ of the diffeomorphism $f$ is called \textit{hyperbolic} if absolute values of all the eigenvalues of the Jacobi matrix $\left(\frac{\partial f^{m_p}}{\partial x}\right)\vert_{p}$ are not unit. If all eigenvalues by  modulo are less than (greater than) one, then $p$ is called a {\it sink} {\it (source) point}. The sink and source points are called {\it nodes}. If a hyperbolic periodic point is not {\it nodal}, then it is called {\it saddle point}.

If $f$ has a finite number of periodic points, all of them are hyperbolic, then the hyperbolic structure of the periodic point $p$ implies that its {\it stable} $$W^s_p=\{x\in M^n:\lim\limits_{k\to+\infty}d(f^{km_p}(x),p)=0\}$$ and  {\it unstable} $$W^u_p=\{x\in M^n:\lim\limits_{k\to+\infty}d(f^{-km_p}(x),p)=0\}$$ manifolds are smooth submanifolds,  diffeomorphic to $\mathbb R^{q_p}$ and $\mathbb R^{n-q_p}$, respectively, where $q_p$ is the number of eigenvalues of the Jacobi matrix with the absolute value greater than 1 ({\it Morse index of the point $p$}). The stable and the unstable manifolds are called {\it invariant manifolds}. 

A number $\nu_p=+1\,(-1)$ is called an {\it orientation type} of the point $p$ if the map $f^{m_{p}}|W^{u}_{p}$ preserves (changes) the orientation.

A connected component $\ell^u_p$ ($\ell^s_p$) of the set $W^u_p\setminus p$ ($W^s_p\setminus p$) is called an {\it unstable (stable) separatrix} of the point $p$.

Diffeomorphism $f: M^{n}\rightarrow M^{n}$ defined on a smooth closed connected orientable $n$-dimensional manifold ($n\geq{1}$) $M^{n}$ is called {\it a Morse-Smale  diffeomorphism} if
\begin{enumerate}
 \item its nonwandering set $\Omega_{f}$ consists of a finite number of hyperbolic orbits;
 \item intersection of the invariant manifolds $W^{s}_{p}$, $W^{u}_{q}$ is transversal for any nonwandering
points $p,q$. 
 \end{enumerate}

Denote by $MS(M^{n})$ the set of orientation-preserving Morse-Smale diffeomorphisms defined on an orientable $n$-manifold $M^n$.

\begin{proposition}[\cite{GrMePo2016}, Theorem 2.1.1.]\label{th_2.1.1.}
Let $f\in MS(M^{n})$. Then
\begin{enumerate}
\item $M^{n} = \underset{p\in\Omega_{f}}{\bigcup}W^{u}_{p}$;
\item $W^{u}_{p}$ is a smooth submanifold of the manifold $M^{n},$ diffeomorphic to $\mathbb{R}^{q_{p}}$ for any periodic point $p\in\Omega_{f}$;
\item $cl(\ell^{u}_{p})\setminus(\ell^{u}_{p}\cup p) = \underset{r\in\Omega_{f}:\ell^{u}_{p}\cap W^{s}_{r}\neq\emptyset}{\bigcup}W^{u}_{r}$ for any unstable separatrix $\ell^{u}_{p}(\ell^{s}_{p})$ of periodic point $p\in\Omega_{f}$.
\end{enumerate}
\end{proposition}
If $\sigma_{1}$, $\sigma_{2}$ different saddle periodic points of the diffeomorphism $f\in MS(M^{n})$, for which $W^{s}_{\sigma_{1}}\cap W^{u}_{\sigma_{2}}\neq\emptyset$, then the intersection $W^{s}_{\sigma_{1}}\cap W^{u}_{\sigma_{2}}$ is called the {\it heteroclinic intersection}. The path connected components of a heteroclinic intersection are called {\it heteroclinic points} if their dimension is 0, {\it heteroclinic curves} if their dimension is 1, and {\it heteroclinic manifolds} if their dimension is greater than 1.

The diffeomorphism $f \in MS(M^n)$ is called {\it gradient-like} if from the condition $W^{u}_{\sigma_1} \cap W^{u}_{\sigma_2} \neq \emptyset$ for different points $\sigma_1, \sigma_2 \in \Omega_f$ it follows that $\dim W^{u}_{\sigma_1}< \dim W^{u}_{\sigma_2}$. This is equivalent to the absence of heteroclinic points for the diffeomorphism $f$.

\begin{proposition}[\cite{GrMePo2016}, Proposition 2.1.3.]\label{odi} If a separatrix $\ell^u_\sigma$ of a saddle point $\sigma$ of a diffeomorphism $f\in MS(M^{n})$ does not participate in the heteroclinic intersection, then there is a unique sink point $\omega$ such that $$cl(\ell^u_\sigma)=\sigma\cup\ell^u_\sigma\cup\omega.$$ At the same time, $cl(\ell^u_\sigma)$ is homeomorphic to the segment if $q_p=1$ and homeomorphic to the sphere $\mathbb S^{q_p}$ if $q_p>1$.
\end{proposition}

Let's put $\hat{W}^{u}_{p}=(W^{u}_{{p}}\setminus {p})/f^{m_p}$ and denote by $p_{\hat{{W}}^{u}_{{p}}}:W^{u}_{{p}}\setminus{p}\rightarrow\hat{{W}}^{u}_{{p}} $ natural prection.
\begin{proposition}[\cite{GrMePo2016}, Theorem 2.1.3]\label{T_2.1.3.}
The projection $p_{\hat{W}^{u}_{{p}}}$ is a covering that induces the structure of a smooth $q_{p}$-manifolds on the space of orbits $\hat{W}^{u}_{{p}}$. At the same time:
\begin{itemize}
\item for $q_p=1, \nu_p=-1$ the space $\hat{W}^{u}_{{p}}$ is homeomorphic to a circle;
\item for $q_p=1, \nu_p=+1$ the space $\hat{W}^{u}_{{p}}$ is homeomorphic to a pair of circles;
\item for $q_p=2$, $\nu_p=-1$ the space $\hat{W}^{u}_{{p}}$ is homeomorphic to the Klein bottle;
\item for $q_p=2$, $\nu_p=+1$ the space $\hat{W}^{u}_{{p}}$ is homeomorphic to a two-dimensional torus;
\item for $q_p\geqslant 3, \nu_p=-1$ space $\hat{{W}}^{u}_{p}$ is homeomorphic to the generalized Klein bottle $\mathbb{S}^{q_p-1}\tilde{\times}\mathbb{S}^{1}$;
\item for $q_p\geqslant 3, \nu_p=+1$ the space $\hat{W}^{u}_{{p}}$ is homeomorphic to $\mathbb{S}^{q_p-1}\times\mathbb{S}^{1}$.
\end{itemize}
\end{proposition}

Let $f\in MS(M^n)$. Denote by $\Omega^0_f,\,\Omega^1_f,\,\Omega^2_f$ the set of sinks, saddles and sources of diffeomorphism $f$. We divide the set $\Omega^1_f$ into two disjoint subsets $\Sigma_A$ and $\Sigma_R$ such that the sets $$A=\Omega^0_f\cup W^u_{\Sigma_A},\,R=\Omega^2_f\cup W^s_{\Sigma_R}$$ are closed and invariant. By construction, the sets $A,\,R$ contain all periodic points of the diffeomorphism $f$. The largest dimension of the unstable (stable) manifold of periodic points from $A$($R$) is called {\it dimension of $A\,(R)$}.

\begin{proposition}[\cite{GrMePoZh}, Theorem 1]\label{AfR} 
Let $f\in MS(M^n)$. Then the set $A$ (respectively $R$) is an attractor (repeller) of the diffeomorphism $f$. Moreover, if the dimension of the attractor $A$ (repeller $R$) $\leqslant n-2$, then the repeller $R$ (attractor $A$) is connected.
\end{proposition}

Following \cite{GrMePoZh}, we will call $A$ and $R$ a {\it dual attractor and repeller} of the Morse-Smale diffeomorphism $f\in MS(M^n)$, and the set $V=M^n\setminus(A\cup R)$ -- {\it a characteristic set}. Denote by $$\hat V=V/f$$ the set of orbits of the action of the group $F=\{f^k,k\in\mathbb{Z}\}$ on the manifold $V$ -- {\it characteristic space}, which coincides with the set of orbits of the diffeomorphism $f$ on $V$. Let $$p_{\hat V}:V\to\hat V$$ be a natural projection that matches the point $x\in V$ with its orbit by virtue of the diffeomorphism $f$ and endows the set $\hat V$ with a factortopology.

\begin{proposition}[\cite{GrMePoZh}, Theorem 2]\label{Vf} For any dual pair of attractor-repeller $A,R$ of the Morse-Smale diffeomorphism $f\in MS(M^n)$ the following is true:
\begin{itemize}
\item the characteristic space $\hat V$ is a closed smooth orientable $n$-manifold, whose each connected component is either irreducible or homeomorphic to $\mathbb{S}^{n-1}\times\mathbb{S}^1$;
\item projection $p_{\hat V}:V\to\hat V$ is a cover;
\item a map $\eta_{\hat V}$, which assigns to each homotopy class $[\hat c]$ of loops $\hat c\subset\hat V$ closed at a point $\hat x$ an integer $n$ such that lifting the loop $\hat c$ by $V$ connects some point $x\in p^{-1}_{\hat V}(\hat x)$ with a point
$f^n(x)$, is a homomorphism on the fundamental group of each  connected component of the space $\hat V$;
\item if the dimension of the attractor $A$ and the repeller $R$ $\leqslant n-2$, then $V,\hat V$ are connected and the map $\eta_{\hat V}:\pi_1(\hat V)\to\mathbb Z$ is an epimorphism.
\end{itemize}
\end{proposition}

A submanifold  $\hat X\subset\hat V$ is called {\it $\eta_{\hat V}$-essential}, if $\eta_{\hat V}(i_{\hat X*}(\pi_1(\hat X))\neq \{0\}$.

Let $U_A$ be a trapping neighborhood of an attractor $A$ of a Morse-Smale diffeomorphism $f:M^n\to M^n$ and $R$ be the dual to it repeller. Let $F_A=U_A\setminus f(U_A)$, then $cl(F_A)$ is the fundamental domain of the  diffeomorphism $f$ restriction to $V$. 
Suppose $\hat V_A=cl(F_A)/f$, then $\hat V_A$ is a smooth closed $n$-manifold obtained from $cl(F_A)$ by identifying boundaries due to the diffeomorphism $f$. 
Denote by $p_A: cl(F_A)\to\hat V_A$ the natural projection. 

Consider the family $E_f\in Diff(M^n)$ of diffeomorphisms such that ${\Omega}_{f'}={\Omega}_{f}$ for any diffeomorphism $f'\in E_f$ and the diffeomorphism $f'$ coincides with the diffeomorphism $f$ on $U_A$ and in some neighborhood of  $R$. 

For any diffeomorphism $f'\in E_f$, we put $\hat l^s_{f'}=p_A(W^s_{\Sigma_A}\cap F_A)$ and $\hat l^u_{f'}=p_A(W^u_{\Sigma_R}\cap  F_A)$.

\begin{proposition}[\cite{PoSh}, Lemma 1]\label{IsotBan}
Let $\hat h: \hat V_A\to \hat V_A$ be an isotopic to identity diffeomorphism. Then there exists a smooth arc ${\varphi}_{t}\subset E_f$ such that ${\varphi}_{0}=f, {\varphi}_{1}=f'$ and $\hat l^u_{f'}=\hat h (\hat l^u_f)$, $\hat l^s_{f'}=\hat l^s_f$.
\end{proposition}

\subsection{Classification of Morse-Smale 3-diffeomorphisms}

Let $f\in MS(M^3)$. Let's put $$A_f=W^u_{\Omega_0\cup\Omega_1},\,R_f=W^s_{\Omega_2\cup\Omega_3},\,V_f=M^3\setminus(A_f\cup R_f).$$ By proposition \ref{Vf} set $A_f~(R_f)$ is a connected attractor (repeller) whose topological dimension is less than or equal to 1, the set $V_f$ is a connected 3-manifold and $$V_f=W^s_{A_f\cap\Omega_f}\setminus A_f=W^u_{R_f\cap\Omega_f}\setminus R_f.$$ Moreover, the space $\hat V_f=V_f/f$ is a connected closed orientable 3-manifold and the  natural projection $p_{_f}:V_f\to\hat V_f$ induces an epimorphism  
$\eta_{_f}:\pi_1(\hat V_f)\to\mathbb Z$, attributing to each homotopy class $[c]\in\pi_1(\hat V_f)$ of a closed curve $c\subset\hat V_f$  an integer $n$ such that a lift of $c$ connects some point $x\in V_f$ with the point
$f^n(x)$. Let's put $$\hat{L}^s_f=p_{_f}(W^s_{\Omega_1}\setminus A_f),\,\hat{L}^u_f=p_{_f}(W^u_{\Omega_2}\setminus R_f).$$ 

Set $S_{f}=(\hat V_{f},\eta_{_{f}},\hat{L}^s_{f},\hat{L}^u_{f})$ is called a {\it scheme} of the diffeomorphism $f\in MS(M^3)$.

\begin{proposition}[\cite{BoGrPo2019}, Theorem 1]\label{clms3} Diffeomorphisms $f,\,f'\in MS(M^3)$ are topologically conjugate if and only if their schemes are equivalent, that is, there is a homeomorphism $\hat\varphi:\hat V_{f}\to\hat V_{f'}$ such that

1) $\eta_{_f}=\eta_{_{f'}}\hat\varphi_*$;

2) $\hat\varphi(\hat L^s_f)=\hat L^s_{f'},\,\hat\varphi(\hat L^u_f)=\hat L^u_{f'}$.
\end{proposition}

To solve the realization problem 
it is necessary to identify the set of all abstract schemes that can be implemented by a Morse-Smale diffeomorphism. 

Let $\hat V$ be a simple smooth 3-manifold whose fundamental group admits an epimorphism $\eta:\pi_1(\hat V)\to\mathbb Z$, $\hat{\ell}\subset\hat V$ be an $\eta$-essential  smooth torus and $N_{\hat{\ell}}\subset\hat V$ is its tubular neighborhood. Let $\hat Y=\mathbb D^2\times\mathbb S^1$ and $\hat\mu$ be a meridian of the solid torus $\hat Y$ (closed curve contractible on $\hat Y$ and essential on $\partial\hat Y$) and  
$\zeta_{_\ell}:\partial\hat Y\times\mathbb S^0\to \partial N_{\hat{\ell}}$ be a diffeomeomorphism such that  
$\eta(\zeta_{_\ell}(\hat{\mu}\times\mathbb S^0))=0$. It is said that the space $\hat V_{\hat{\ell}}=
(\hat V\setminus int\,N_{\hat{\ell}})\cup_{\zeta_{_\ell}}
(\hat{Y}\times\mathbb S^0)$ {\it is obtained from the manifold $\hat V$ by a cut-gluing operation  along the torus} $\hat{\ell}$.

Structure of a smooth closed 3-manifold on the set $\hat V_{\hat{\ell}}$ 
 induced by the natural projection $p_{_{\hat{\ell}}}:(\hat V\setminus int\,N_{\hat{\ell}})\sqcup
(\hat{Y}\times\mathbb S^0)\to\hat V_{\hat{\ell}}$. 
Since any homeomorphism of the boundary of the solid torus that translates meridian to meridian can be extended to the solid torus \cite{Ro}, the described operation is correctly defined, that is, it does not depend (up to homeomorphism) on the choice of the tubular neighborhood ${N}_{\hat{\ell}}$ and the homeomorphism $\zeta_{_\ell}$.

Similarly, a {\it cut-gluing operation  along an $\eta$-essential smooth Klein bottle} $\hat{\ell}\subset\hat V$ is defined and it consists of a gluing the solid torus $\hat Y$ to the boundary of the manifold 
$\hat V\setminus int\,N_{\hat{\ell}}$. Also, the cut-gluing operation is generalized to the set $\hat L\subset\hat V$, which is a disjoint union of smooth $\eta$-essential tori and Klein bottles, we will denote by $V_{\hat L}$ the manifold obtained as a result of such an operation. 

For a gradient-like diffeomorphism $f\in MS(M^3)$ each connected component $\hat\ell^s\,(\hat\ell^u)$ of the sets $\hat L^s_{f}\,(\hat L^u_f)$ is either a torus or a Klein bottle,  $\eta_{_f}$-essentially embedded into the manifold $\hat V_f$. 

The scheme of any gradient-like diffeomorphism $f\in MS(M^3)$ is an abstract schema in the sense of the following definition.

A collection $S=(\hat V,\eta,\hat{L}^s,\hat{L}^u)$ is called an {\it abstract scheme} if:

1) $\hat V$ is a simple manifold whose fundamental group admits an epimorphism $\eta:\pi_1(\hat V)\to \mathbb Z$;

2) the sets $\hat{L}^s,\,\hat{L}^u\subset\hat V$ are transversally intersecting disjoint unions of smooth    
$\eta$-essential tori and Klein bottles;

3) each connected component of the  manifolds $\hat V_{\hat{L}^s},\,\hat V_{\hat{L}^u}$ is homeomorphic to the manifold $\mathbb S^2\times\mathbb S^1$.

\begin{proposition}[\cite{BoGrPo2017}, Theorem 1]\label{rems3} For any abstract scheme $S=(\hat V,\eta,\hat{L}^s,\hat{L}^u)$ there is a gradient-like diffeomorphism $f\in MS(M^3)$ whose scheme $S_f$ is equivalent to the scheme $S$.
\end{proposition}

\subsection{Topology of 3-manifolds admitting Morse-Smale diffeomorphisms with a given structure of a non-wandering set}
Let $f\in MS(M^3)$. Let's say $$g_{_f}=\frac{r_{_f}-l_{_f}+2}{2},$$ where $r_{_f}$ is the number of saddle points and $l_{_f}$ is the number of nodal periodic points of the diffeomorphism $f$. According to \cite{GrMePo2016}, the number $g_{_f}$ is a non-negative integer for any diffeomorphism $f\in MS(M^3)$.

If $f\in MS(M^3)$ is a gradient-like diffeomorphism. According to proposition \ref{odi}, the closure $cl(\ell^u_\sigma)$ of any one-dimensional unstable separatrix $\ell^u_\sigma$ of the saddle point
$\sigma$ of the diffeomorphism $f$ is homeomorphic to a segment that consists of this separatrix and two points: $\sigma$ and some sink  $\omega$. Let ${L}_\omega$ be a union of unstable one-dimensional separatrices of saddle
points that contain $\omega$ in their closures. According to proposition \ref{T_2.1.3.}, $W^s_\omega$ is homeomorphic to $\mathbb R^3$ and the set $L_\omega\cup\omega$ is a union of simple arcs with a single common point $\omega$, then by analogy with a frame of arcs in $\mathbb R^3$, $L_\omega\cup\omega$ is called {\it a frame of one-dimensional unstable separatrices}.  

According to \cite{GrMeZh2003}, a frame of separatrix $L_\omega\cup\omega$ is called {\it tame} if there is a homeomorphism $\psi_\omega:W^
{s}_\omega\to \mathbb R^3$ such that $\psi_\omega
(L_\omega\cup\omega)$ --- a frame of rays in $\mathbb R^3$. Otherwise, the separatrix frame is called {\it wild} (see Fig. \ref{miwi}).
\begin{figure}\centering{\includegraphics[scale=
0.7]{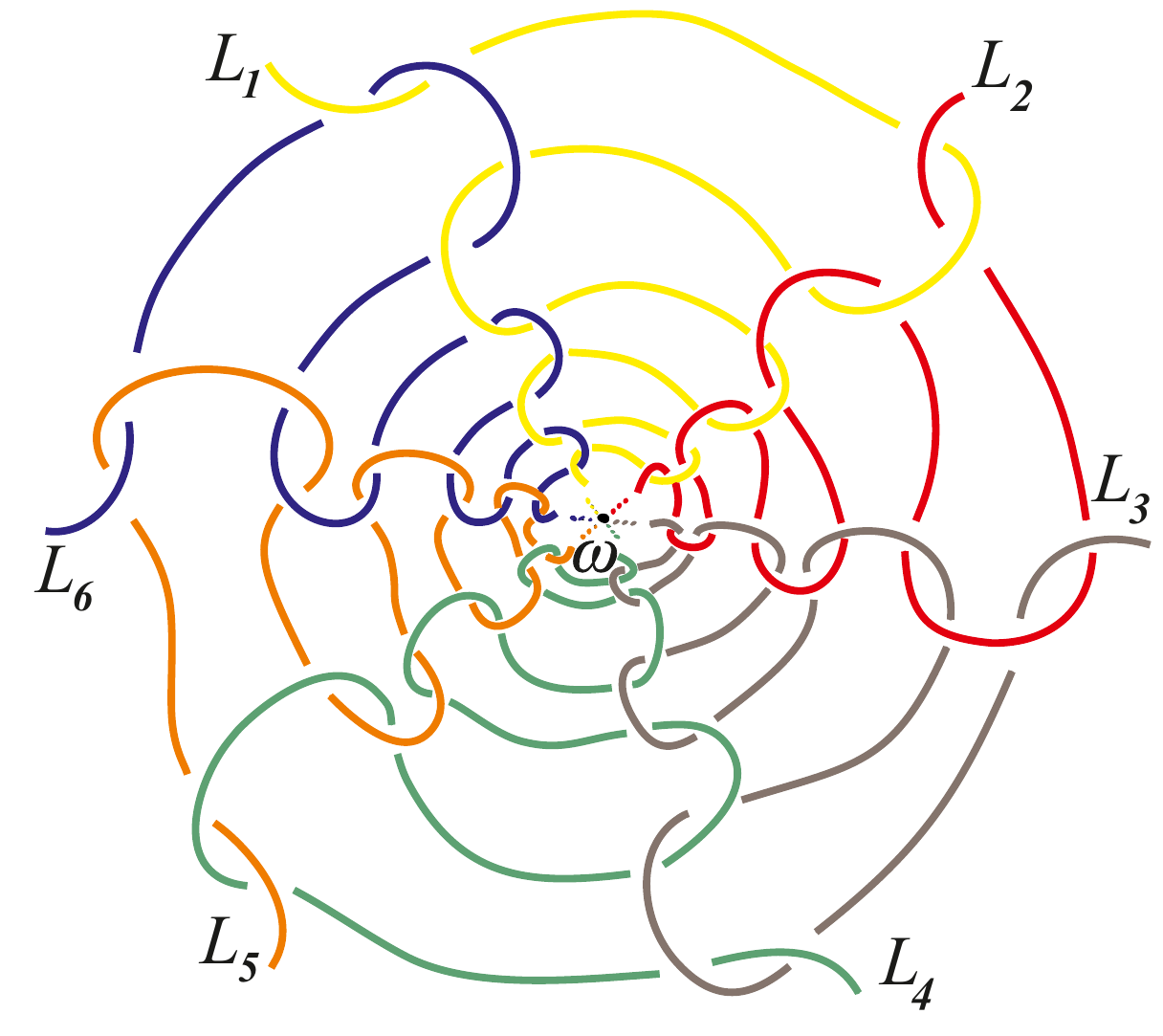}}
		\caption{\small Wild frame  of separatrix in which each separatrix is tame} \label{miwi}
\end{figure}

If $\alpha$ is the source of the diffeomorphism $f$, then the tame (wild)
bundle $L_\alpha$ of one-dimensional stable separatrices is similarly defined.

\begin{proposition}[\cite{GrMeZh2003}, Theorem 4.1] \label{Heeg} If all frames of one-dimensional separatrices of a  gradient-like diffeomorphism $f\in MS(M^3)$ are tame, then the ambient manifold $M^3$ admits a Heegaar splitting of the genus $g_f$.
\end{proposition}

\begin{proposition}[\cite{BoGrMePe2002}, Theorem 1] \label{svaz} Let $f\in MS(M^3)$ be a Morse-Smale diffeomorphism without heteroclinic curves. Then the following statements are true: 

1) if $g_{_f}=0$, then $M^3$ -- 3-sphere; 

2) if $g_{_f}>0$, then $M^3$ --  connected sum of $g_{_f}$ copies of $\mathbb S^2\times\mathbb S^1$. 

Conversely, for any non-negative integers $r, l, g$ such that the number $g=\frac{r-l+2}{2}$ is an integer and non-negative, there is a diffeomorphism $f\in MS(M^3)$ without heteroclinic curves, with the following properties: 

a) $M^3$ -- 3-sphere if $g=0$ and $M^3$ -- connected sum of $g$ copies of $\mathbb S^2\times\mathbb S^1$ if $g>0$;

b) the non-wandering diffeomorphism set $f$ consists of $r$ saddle and $l$ node points. 
\end{proposition}

Recall that {\it lens space} is defined as a gluing of two solid tori by means of a homeomorphism of their boundaries and is denoted by $L_{p,q},\,p,q\in\mathbb Z$, where ${\langle p,q\rangle}$ is the homotopy type of the image of a meridian with respect to the gluing homeomorphism. Some well-known $3$-manifolds are actually lens spaces, for example, the three-dimensional sphere $\mathbb S^3=L_{1,0}$, the manifold $\mathbb S^2\times\mathbb S^1=L_{0,1}$, the projective space $\mathbb RP^3=L_{1,2}$. 

\begin{proposition}[\cite{GrMeZh2003}, Theorem 6.1] \label{Lpq} Let $f:L_{p,q}\to L_{p,q}$ be a Morse-Smale diffeomorphism whose non-wandering set consists of exactly four points. Then
\begin{enumerate}
\item[1)] $f$ -- gradient-like;
\item[2)] periodic points of the diffeomorphism $f$ have pairwise different Morse indices;
\item[3)] if all frames of one-dimensional separatrices of $f$ are tame, then the wandering set of diffeomorphism $f$ contains at least $p$ of non-compact heteroclinic curves. 
\end{enumerate} 
\end{proposition}

\section{Dynamics of diffeomorphisms of the class $G$}
 
In this section, we establish some dynamic properties of the diffeomorphism $f:M^3\to M^3$ from the class $G$. 

Recall that the class $G$ consists of diffeomorphisms $f\in MS(M^3)$ having exactly four non-wandering points $\omega_f,\sigma_f^{1},\sigma_f^{2},\alpha_f$ with Morse indices $0,1,2,3$, respectively. 

Due to the absence of heteroclinic points in the diffeomorphism $f$, one-dimensional saddle manifolds contain a unique nodal point in their closures (see, sentence \ref{odi}).  Exactly, 
$$cl(W^{u}_{\sigma_f^1})= W^{u}_{\sigma_f^1} \cup \omega_f,\,cl(W^{s}_{\sigma_f^2})=W^{s}_{\sigma_f^2}\cup  \alpha_f.$$
In this case, by  proposition  \ref{T_2.1.3.}, the sets $A_f=cl(W^{u}_{\sigma_f^1}),\,R_f=cl(W^{s}_{\sigma_f^2})$ are pairwise disjoint topologically embedded circles (see Fig.  \ref{af2}, \ref{neor}, \ref{pic2}, \ref{pic3}), possibly wild at the nodal points.
\begin{figure}[h]
\center{\includegraphics[width=0.7\linewidth]{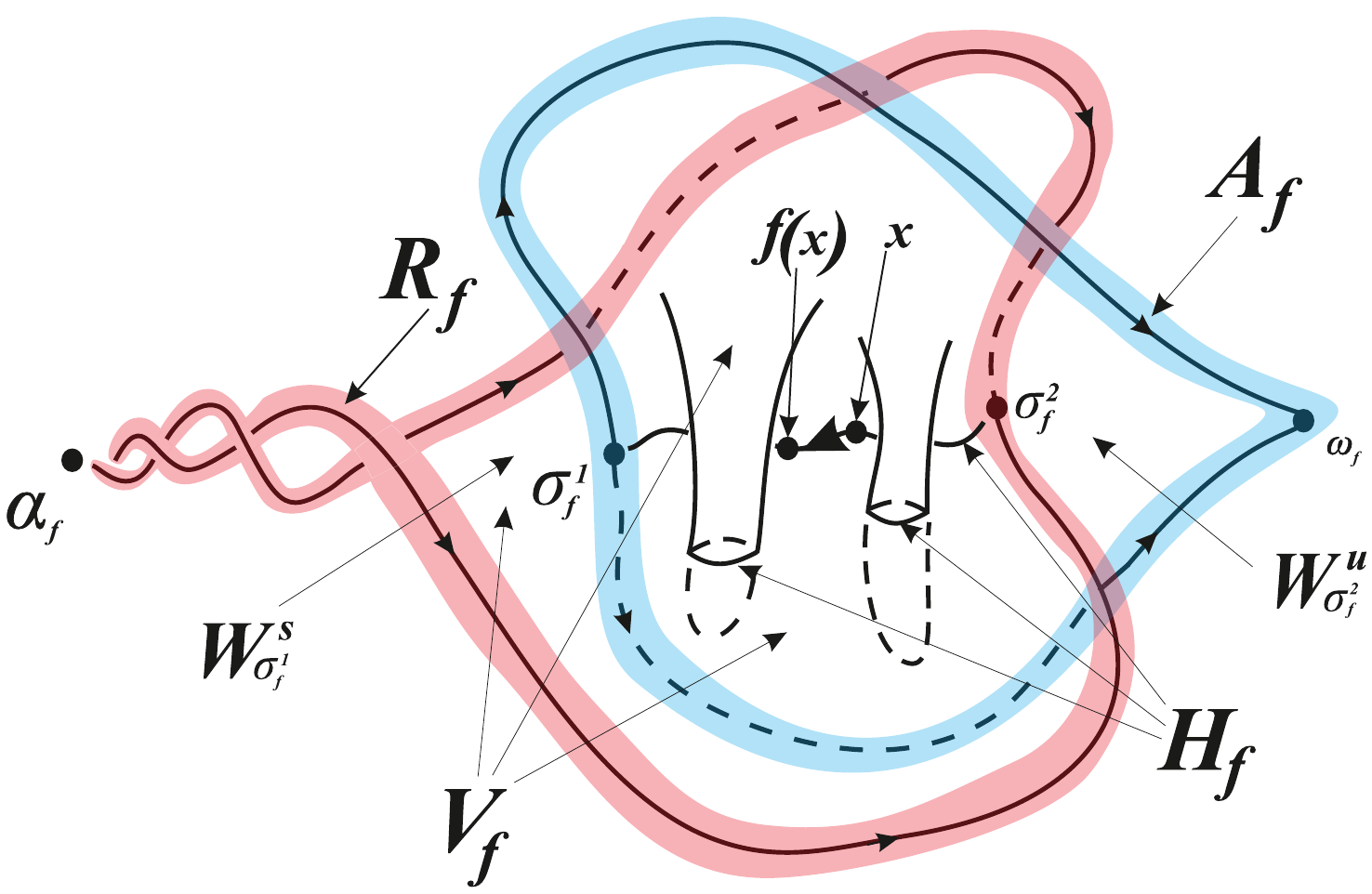}}
\caption{\small Phase portrait of a diffeomorphism $f\in G$ with a non-empty set $H_f$}
\label{pic2}
\end{figure} 
Recall that  
$$H_f=W^s_{\sigma_f^1}\cap W^u_{\sigma_f^2}.$$ If the set $H_f$ is not empty, then, by proposition \ref{th_2.1.1.}, $$cl(W^{s}_{\sigma_f^1})=W^{s}_{\sigma_f^1} \cup R_f,\,cl(W^{u}_{\sigma_f^2})=W^{u}_{\sigma_f^2} \cup A_f.$$ Otherwise,
according to proposition \ref{odi}, the sets $$cl(W^{s}_{\sigma_f^1})=W^{s}_{\sigma_f^1} \cup\alpha_f,\,cl(W^{u}_{\sigma_f^2})=W^{u}_{\sigma_f^2}\cup\omega_f$$ are topologically embedded disjoint two-dimensional spheres (see Fig. \ref{pic3}).
\begin{figure}[h]
\center{\includegraphics[width=0.55\linewidth]{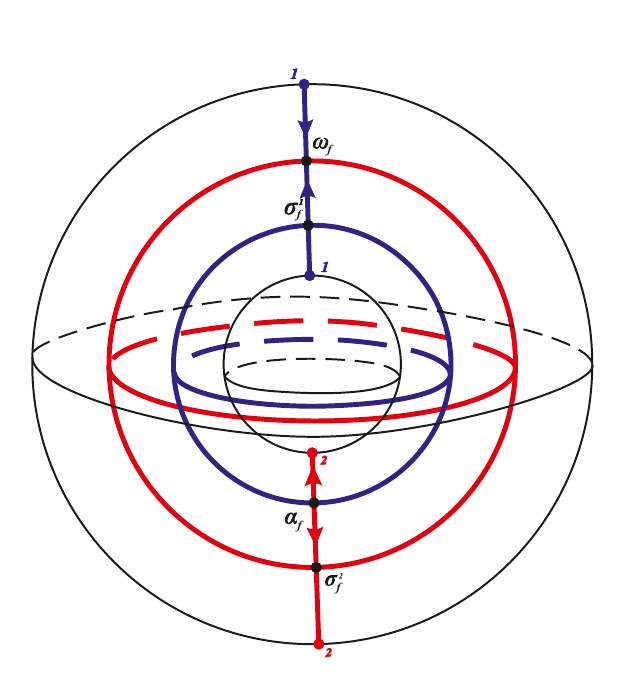}}
\caption{\small Phase portrait of a diffeomorphism $f\in G$ with an empty set $H_f$}
\label{pic3}
\end{figure}

\subsection{Consistent neighborhoods system}
Let $\mathcal N_{1}=\{(x_1,x_2,x_3)\in\mathbb{R}^3: x_1^2(x_2^2+x_3^2)\leqslant 1\}$ and $\mathcal N_{2}=\{(x_1,x_2,x_3)\in\mathbb{R}^3:  (x_1^2+x_2^2)x_3^2\leqslant 1\}$.  
Define in the neighborhood of $\mathcal N_{1}$ a pair of transversal foliations $\mathcal{F}^u_1, \mathcal{F}^s_{1}$ as follows:   
$$\mathcal{F}^u_1=\bigcup\limits_{(c_{2},c_3)\in Ox_2x_3}\{(x_1,x_2,x_3)\in \mathcal N_{1}: (x_{2},x_3)=(c_{2},c_3)\},$$ $$\mathcal{F}^s_{1}=\bigcup\limits_{c_1\in Ox_1}\{(x_1,x_2,x_3)\in \mathcal N_{1} : x_1=c_1\}.$$

Define in the neighborhood of $\mathcal N_{2}$ a pair of transversal foliations $\mathcal{F}^u_2, \mathcal{F}^s_{2}$ as follows:  
$$\mathcal{F}^u_2=\bigcup\limits_{c_3\in Ox_3}\{(x_1,x_2,x_3)\in \mathcal N_{3} : x_3=c_3\},$$ $$\mathcal{F}^s_{2}=\bigcup\limits_{(c_{1},c_2)\in Ox_1x_2}\{(x_1,x_2,x_3)\in \mathcal N_{3} : (x_{1},x_2)=(c_{1},c_2)\}.$$

We define diffeomorphisms 
$\nu_i:\mathbb R^3\to\mathbb R^3$ by formulas: $$\nu_1(x_1,x_2,x_3)=\left(2x_1,\frac{x_2}{2},\frac{x_3}{2}\right),\,\nu_2=a_1^{-1}.$$
Note that for $i\in\{1,2\}$, the set $\mathcal N_{i}$ is invariant with respect to diffeomorphism
$\nu_i$, which translates leaves of the foliation
$\mathcal{F}^u_i$ ($\mathcal{F}^s_{i}$) into leaves of the same foliation. 

By \cite{BoGrLaPo2019}, the saddle point $\sigma_f^i$ of the diffeomorphism $f\in G$ has {\it a linearizing neighborhood} $N_f^{i}$ equipped with the homeomorphism ${\mu}_{i}:N_f^{i}\to {\mathcal N}_{i}$, conjugating the diffeomorphism $f\vert_{{N}_f^{i}}$ with the diffeomorphism $\nu_i|_{{\mathcal N}_{i}}$ and being a diffeomorphism on $N_f^i\setminus(W^s_{\sigma_f^i}\cup W^u_{\sigma_f^i})$. Foliations $\mathcal{F}^u_{i}, \mathcal{F}^s_{i}$ are induced by the homeomorphism
${\mu}_{i}^{-1}$, $f$-invariant foliations of ${F}^u_{i}, {F}^s_{i}$ 
on the linearizing neighborhood  $N_f^{i}$. For any point $x\in N_f^i$ we will denote by ${F}^u_{i,x}$ (${F}^s_{i,x}$) a unique  leaf of the foliation ${F}^u_{i}$ (${F}^s_{i}$) passing through the point $x$.

If the set $H_f$ is empty, then the set $N_f$ of disjoint  linearizing neighborhoods $N_f^1,N_f^2$ of saddle points of the diffeomorphism $f$ is called a {\it  consistent neighborhoods system}, and the foliations $F^s_i, F^u_i$ $(i=1,2)$ -- {\it consistent}.

If $H_f\neq\emptyset$, then we choose $f$-invariant tubular neighborhood $N_{H_f}\subset M^3$ of curves of the set $H_f$, equipped with $f$-invariant $C^{1,1}$-foliation $F$, consisting of two-dimensional disks, transversal to $H_f$. For any point $x\in N_{H_f}$, we will denote by ${F}_{x}$ a unique leaf of the foliation ${F}$ passing through the point $x$. 
\begin{figure}[h!]
	\centerline
	{\includegraphics[width=16 cm]{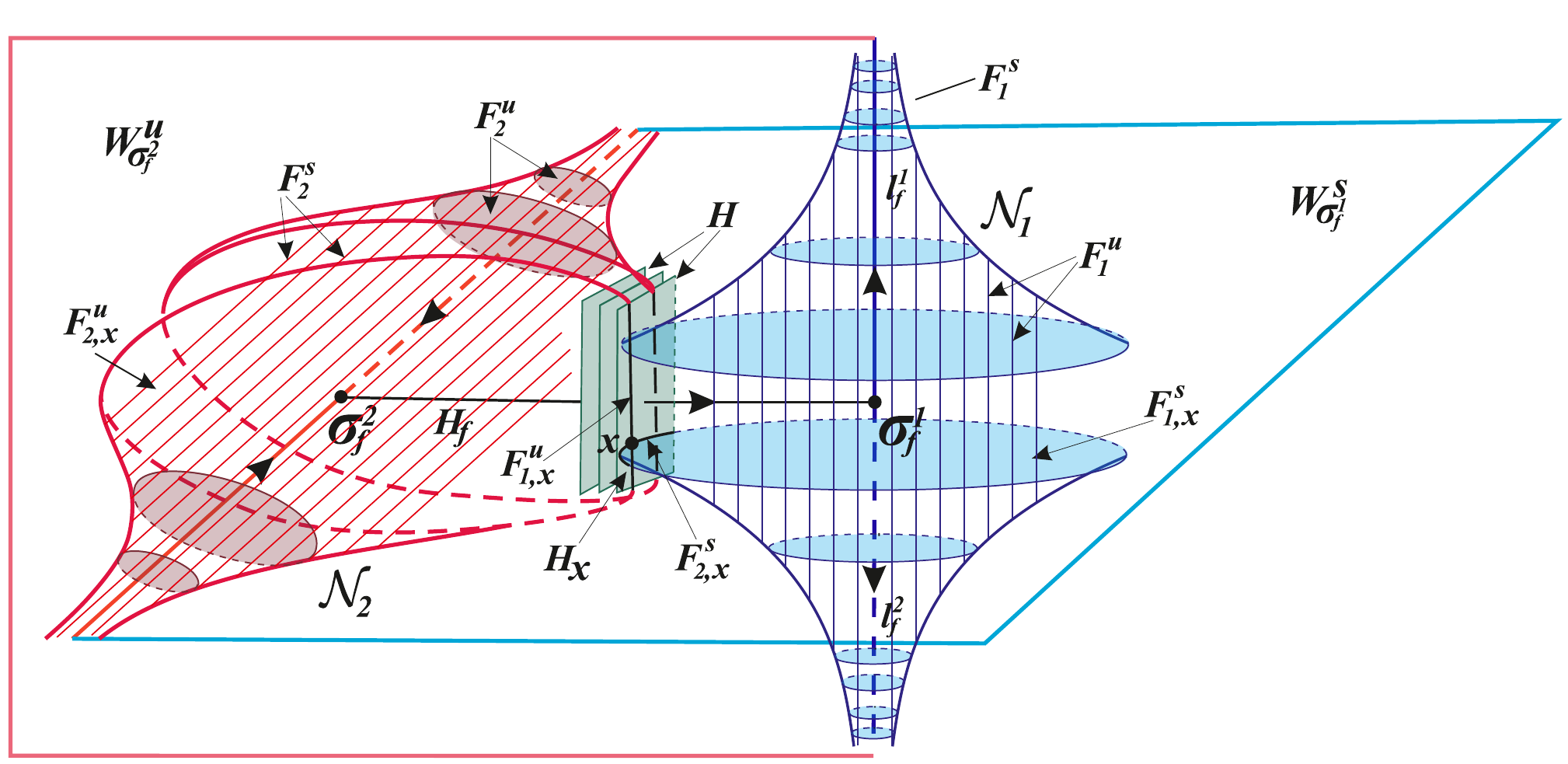}}
	\caption{\small Consistent neighborhoods system}\label{sogl}
\end{figure}

The union $N_f$ of linearizing neighborhoods of $N_f^1,N_f^2$ saddle points of the diffeomorphism $f$ is called a {\it  consistent neighborhoods system}, and the foliations $F^s_i, F^u_i$ $(i=1,2)$, {\it are consistent} if for any point $x\in(N_f^1\cap N_f^2\cap N_{H_f})$ and the leaf $F_x$ of the foliation $F$ passing through the point $x$, the conditions are met (see Fig. \ref{sogl}): $${F}^s_{1,x}\cap F_x={F}^s_{2,x}\cap({N}_f^{1}\cap N_{H_f}),\,\,\,{F}^u_{2,x}\cap F_x={F}^u_{1,x}\cap
({N}_f^{2}\cap N_{H_f}).$$
\begin{proposition}[\cite{BoGrPo2019}, Theorem 1] For any diffeomorphism $f\in G$ there is a consistent neighborhoods system.
\end{proposition}

\subsection{Quotients}
Consider the characteristic spaces $V_{\omega_f}=W^s_{\omega_f}\setminus\omega_f$ and $\hat V_{\omega_f}=V_{\omega_f}/f$. By proposition \ref{T_2.1.3.}, $\hat V_{\omega_f}$ is diffeomorphic to the manifold $\mathbb{S}^2\times\mathbb{S}^1$. By proposition \ref{Vf}, the projection $p_{\omega_f}:V_{\omega_f}\to\hat V_{\omega_f}$ is covering which generate an epimorphism $\eta_{\omega_f}:\pi_1(V_{\omega_f})\to\mathbb Z$. Let's put (see Fig. \ref{Ho}) $$\hat A_f=p_{_{\omega_f}}(A_f).$$
\begin{figure}[h]
\center{\includegraphics
[width=0.45\linewidth]{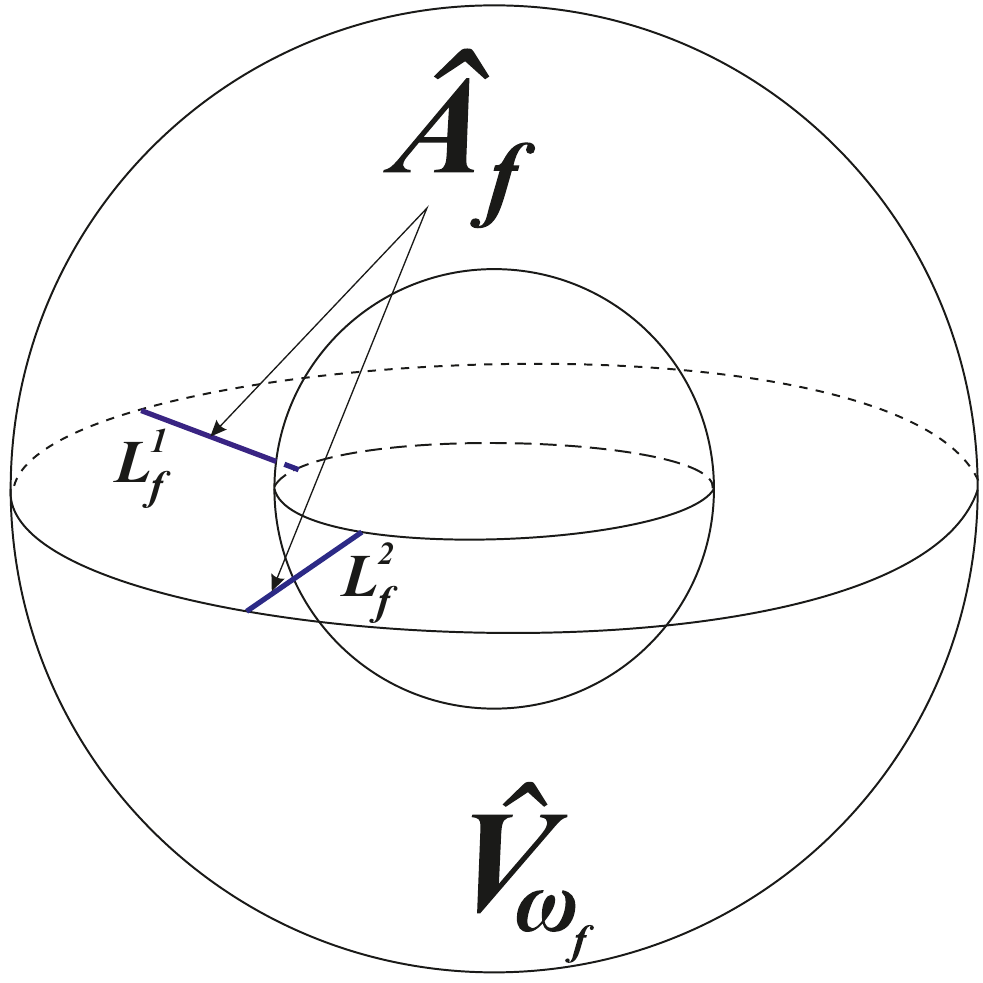}}
\caption{\small Space $\hat V_{\omega_f}$}
\label{Ho}
\end{figure}
By proposition \ref{T_2.1.3.},  $\hat A_f$ consists of a pair of disjoint knots $L^1_f\sqcup L^2_f$ such that the map $i_{L^i_f*}:\pi_1(L^i_f)\to\pi_1(\hat V_{\omega_f}),\,i\in\{1,2\}$ is an isomorphism. Moreover, $$N_{\hat A_f}=p_{_{\omega_f}}(N^1_f\cap V_{\omega_f})$$ is a disjoint union of tubular neighborhoods $N_{L^1_f},\, N_{L^1_f}$ of knots $L^1_f,L^2_f$, accordingly.

\begin{proposition}[\cite{GrMePo2016}, Lemma 4.4]\label{wild} If at least one of the sets $\hat V_{\omega_f}\setminus int\,N_{L^1_f},\,\hat V_{\omega_f}\setminus int\,N_{L^2_f}$ is not homeomorphic to a solid torus, then the manifold $W^u_{\sigma^1_f}$ is wildly embedded into the supporting manifold $M^3$.
\end{proposition}

By proposition \ref{AfR}, the sets $A_{f}$ and $R_{f}$ are dual attractor and repeller, respectively, for the diffeomorphism $f$. Let's put $$V_f=M^3\setminus(A_f\cup R_f).$$
By proposition \ref{Vf}, the characteristic space $\hat{V}_f=V_f/f$ is a smooth simple orientable 3-manifold, and the natural projection $p_{_f}:V_f\to\hat{V}_f$ is a cover inducing an epimorphism $$\eta_{_f}:\pi_1(\hat V_f)\to\mathbb{Z}.$$  Let's put (see Fig. \ref{1_inter}) $$T^s_f=p_{_f}(W^s_{\sigma_1}),\,T^u_f=p_{_f}(W^{u}_{\sigma_{2}}), C_f=p_{_f}(H_f).$$
\begin{figure}[h]
\center{\includegraphics
[width=0.35\linewidth]{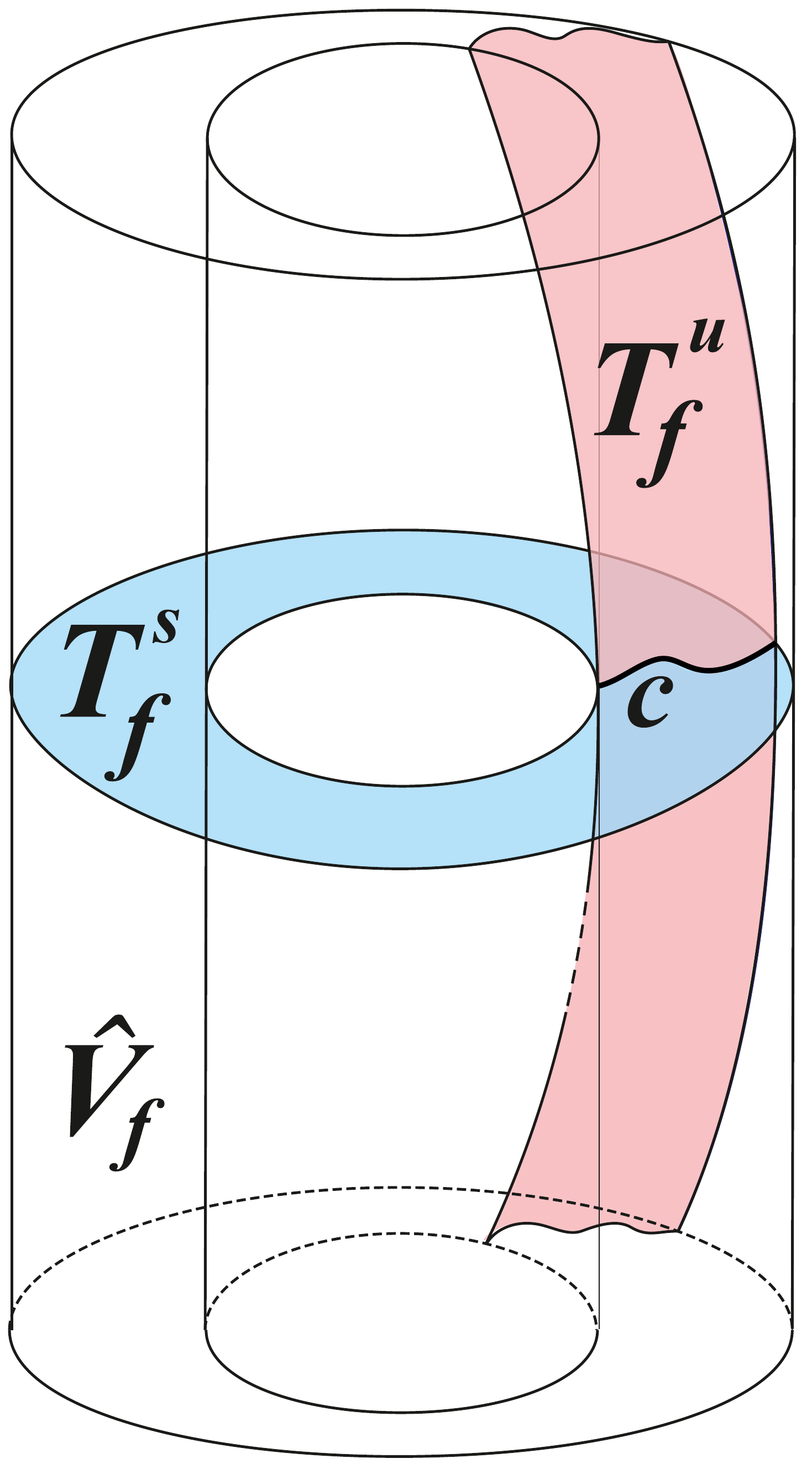}}
\caption{\small Space $\hat{V}_f$}
\label{1_inter}
\end{figure} 
By propositions \ref{T_2.1.3.} and \ref{Vf}, the sets $T^s_f,\,T^u_f$ are smoothly embedded 2-tori in $\hat V_f$ such that $\eta_{_f}(i_{T^s_f*}(\pi_1(T^s_f)))=\eta_{_f}(i_{T^u_f*}(\pi_1(T^u_f)))\cong\mathbb Z$.
Moreover $$N_{T^s_f}=p_{_{f}}(N^1_f\cap V_{f}),\,N_{T^u_f}=p_{_{f}}(N^2_f\cap V_{f})$$  are tubular neighborhoods  of the tori  $T^s_f,T^u_f$, accordingly.

\subsection{The scheme of $f\in G$}
As was mentioned above the collection $$S_f=(\hat V_f,\eta_{_f},T^s_f,T^u_f)$$ is  the scheme of $f\in G$ and, by proposition \ref{clms3}, is a complete invariant of the topological classification. 

The main result of the section is the following lemma.

\begin{lemma}[\cite{PoSh}, Lemma 2]\label{dyn} For any diffeomorphism $f\in G$, the following is true:
\begin{enumerate}
\item the manifold $\hat V_f$ is irreducible and the tori $T^s_f,\,T^u_f$ are incompressible in it;
\item the set $C_f$ consists of a finite number of smoothly embedded closed curves, while $\eta_{_f}([c])=0$ if and only if the curve $c\subset C_f$ is a projection of a compact heteroclinic curve;
\item any curve $c\subset C_f$ such that $\eta_{_f}([c])=0$ is contractible (or non-contractible) simultaneously on both tori $T^s_f,\, T^u_f$. 
\end{enumerate}
\end{lemma}
\begin{proof} Let us prove successively all the statements of the lemma.

1. By virtue of the sentence \ref{AfR}, the manifold $\hat V_f$ is simple. Since the torus $T^s_f$ is $\eta_{_f}$-essential in $\hat V_f$, then it does not lie in a 3-ball. Let's show from the opposite that the torus $T^s_f$ does not bound the solid torus in $\hat V_f$. It follows from proposition \ref{f_1}, that the manifold $\hat V_f$ is not homeomorphic to $\mathbb{S}^2\times\mathbb{S}^1$ and, therefore, is irreducible. Then, by proposition \ref{ex_6}, the torus $T^s_f$ is incompressible in $\hat V_f$. 

If we assume that $\hat V_f\cong\mathbb{S}^2\times\mathbb{S}^1$, then the torus $T^s_f$ bounds a solid torus there by proposition  \ref{f_1}, and therefore $\hat V_f\setminus T^s_f$ consists of two connected components. On the other hand, by proposition  \ref{th_2.1.1.}, $$M^3=W^s_{\omega_f}\cup W^{s}_{\sigma_f^{1}}\cup W^{s}_{\sigma_f^{2}}\cup W^{s}_{\alpha_f}.$$ Then $V_f\setminus W^s_{\sigma_f^1}=W^s_{\omega_f}\setminus A_f$ and, hence, the manifolds $\hat V_f\setminus T^s_f$ and $\hat V_{\omega_f}\setminus\hat A_f$ are homeomorphic. 
Since a one-dimensional submanifold does not divide a manifold of dimension three (see  proposition \ref{sl_1}), the set $\hat V_{\omega_f}\setminus\hat A_{f}$ is connected (see Fig. \ref{Ho}). We got a contradiction with the fact that a connected manifold is homeomorphic to an non-connected one.

2. It follows directly from the definition of the epimorphism $\eta_{_f}$ that $\eta_{_f}([c])=0$ if and only if $c$ is a projection of a compact curve.   

3. Suppose that some curve $c\subset C_f$ is contractible on the torus $T^u_f$ and essential on the torus $T^s_f$. Then, by definition, the torus $T^s_f$ is compressible in $\hat V_f$, which contradicts the proven point 1.  
\end{proof}

Denote by $C^0_f$ a subset of $C_f$ consisting of contractible on $T^u_f$ curves. Let's call the heteroclinic curves from the set $H^0_f=p_f^{-1}(C^0_f)$ {\it inessential}, the remaining heteroclinic curves will be called {\it essential}.
 
\section{Trivialization of the dynamics of diffeomorphisms from $G$}\label{sec-tri}
Recall that for any diffeomorphism $f\in G$, we introduced the concept of a heteroclinic index $I_f$ of $f$ and for an integer $p\geqslant 0$ we denote by $G_p\subset G$ a subset of diffeomorphisms $f\in G$ such that $I_{f}=p$. Note that any essential compact heteroclinic curve $\gamma\subset H_f$ bounds a disk $d_\gamma\subset W^s_{\sigma_f^1}$ containing the saddle $\sigma_f^1$. We will consider any such curve oriented so that when moving along it, the disk $d_\gamma$ remains on the left.
Then for the curve $\gamma$, similarly to a non-compact curve, an orientation $v_\gamma$ is determined.  

The set $H_f$ is called {\it orientable} if it consists only of essential heteroclinic curves with the same orientation. Otherwise, we will call the set $H_f$ {\it not orientable} (see Fig. \ref{neor}). Denote by $G^+_p\subset G_p,\,p\geqslant 0$ a subset of diffeomorphisms $f\in G_p$ with an orientable set $H_f$.
Thus, for any diffeomorphism $f\in G^+_p$, the set $H_f$ is either empty, or consists either only of non-compact or only of compact heteroclinic curves (see Fig. \ref{hecom}).   
\begin{figure}[h]
\center{\includegraphics[width=0.5\linewidth]{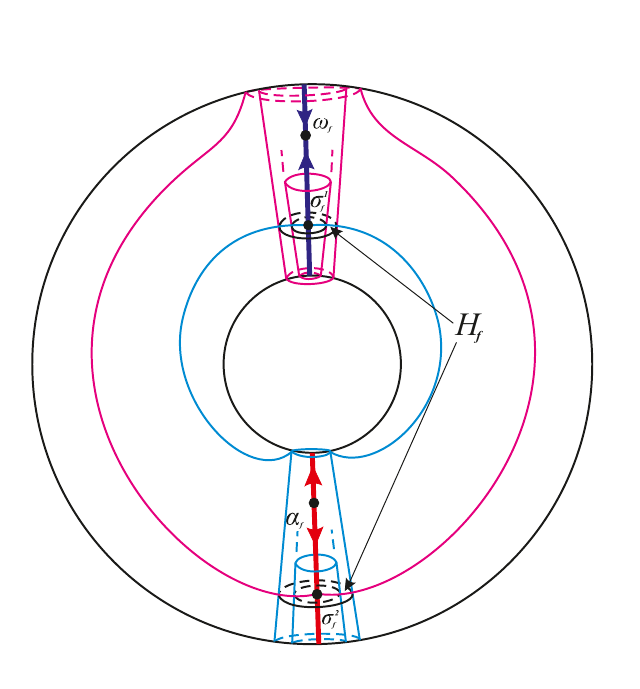}}
\caption{\small Diffeomorphism $f\in G_0^+$ with an orientable set $H_f$ consisting of an infinite set of compact curves}
\label{hecom}
\end{figure}

The main result of this chapter is the proof of the following theorem.   
\begin{theorem}\label{te1}
For any diffeomorphism $f:M^3\to M^3$ from the class $G_p,\,p\geqslant 0$ in the set $Diff(M^3)$ there is an arc connecting the diffeomorphism $f$ with some diffeomorphism $f_+\in G_p^+$.  
\end{theorem}
The proof of the theorem will directly follow from the lemmas \ref{te11}, \ref{te12}, proved below.

\subsection{Disappearance of inessential heteroclinic curves}
Denote by $\tilde G_p\subset G_p$ a subclass of diffeomorphisms $f$ for which the set $H^0_f$ is empty. 

The main result of this section is the proof of the following fact.  

\begin{lemma}\label{te11} For any diffeomorphism ${f}:M^3\to M^3$ from the class ${G}_p$ there exists an arc in the set $Diff(M^3)$ connecting the diffeomorphism $f$ with some diffeomorphism $\tilde f\in\tilde G_p$.
\end{lemma}
\begin{proof} Let $f\in G_p$.  If $H^0_f=\emptyset$, then the lemma is proved. 
\begin{figure}[h]
\center{\includegraphics
[width=0.35\linewidth]{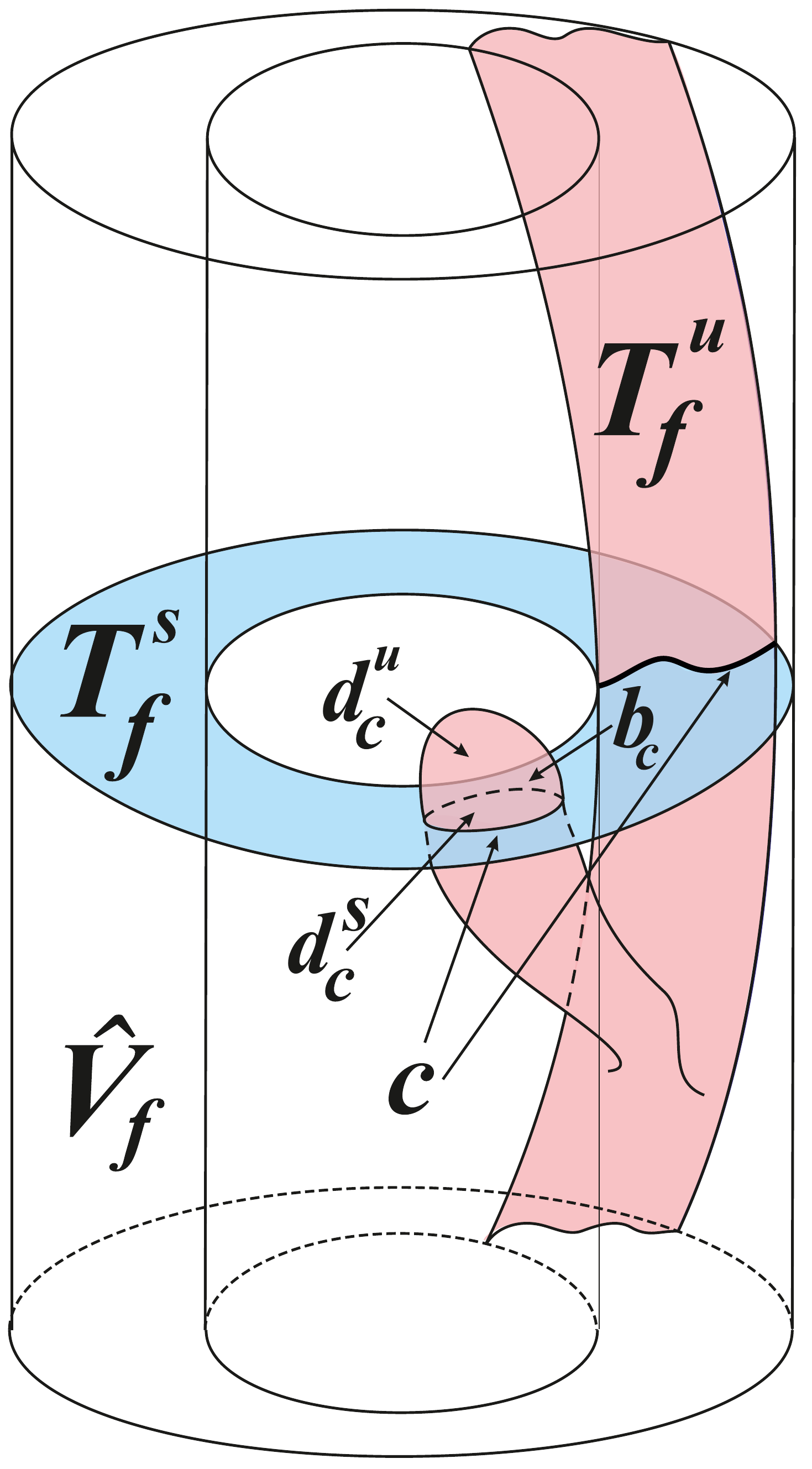}}
\caption{\small Construction of the 3-ball $b_c$}
\label{1-inter}
\end{figure}
Otherwise, by  lemma \ref{dyn}, for any curve $c\subset C^0_f$ there exists a unique disk $d^{s}_{c}$ such that $d^{s}_{c}\subset T^{s}_{f}$, $c=\partial d^{s}_{c}$ and a similar disk $d^{u}_{c}\subset T^{u}_{f}$ curve $c=\partial d^{u}_{c}$ (see Fig. \ref{1-inter}). 

Among the curves of the set $C^0_f$, we choose an {\it innermost curve} $c$, that is, such that $int\,d^s_c\cap C^0_f=\emptyset$. Since $d^s_c\cap d^u_c=c$, the set $d_c^s\cup d_c^u$ is a two-dimensional sphere cylindrical embedded into the manifold $\hat V_f$. By lemma \ref{dyn}, the manifold $\hat V_f$ is irreducible and, therefore, this sphere bounds a three-dimensional ball $b_c$  there. Denote by $T^u_{f,c}$ a two-dimensional torus obtained by smoothing the torus $(T^u_f\setminus d^u_c)\cup d^s_c$. 
Then there is an isotopic to identity diffeomorphism $\hat h:\hat V_f\to\hat V_f$ such that $\hat h(T^u_f)=T^u_{f, C^0_f}$. Then by proposition \ref{IsotBan}  there is an arc ${\zeta}_{t}\subset E_f$ such that ${\zeta}_{0}=f$ and $T^u_{\zeta_1}=T^u_{f, c}$, $T^s_{\zeta_1}=T^s_f$. 

Repeating this process for each innermost curve, we get a required  diffeomorphism $\tilde f\in\tilde G_p$. 
\end{proof}

\subsection{Disappearance of non-orientable heteroclinic curves}
The main result of this section is the proof of the following fact.  

\begin{lemma}\label{te12} For any diffeomorphism ${f}:M^3\to M^3$ from the class $\tilde{G}_p$ there exists an arc in the set $Diff(M^3)$ connecting the diffeomorphism $f$ with some diffeomorphism $f_+\in G^+_p$.
\end{lemma}
\begin{proof} Let ${f}\in\tilde{G}_p$.
\begin{figure}[h]
\center{\includegraphics
[width=0.5\linewidth]{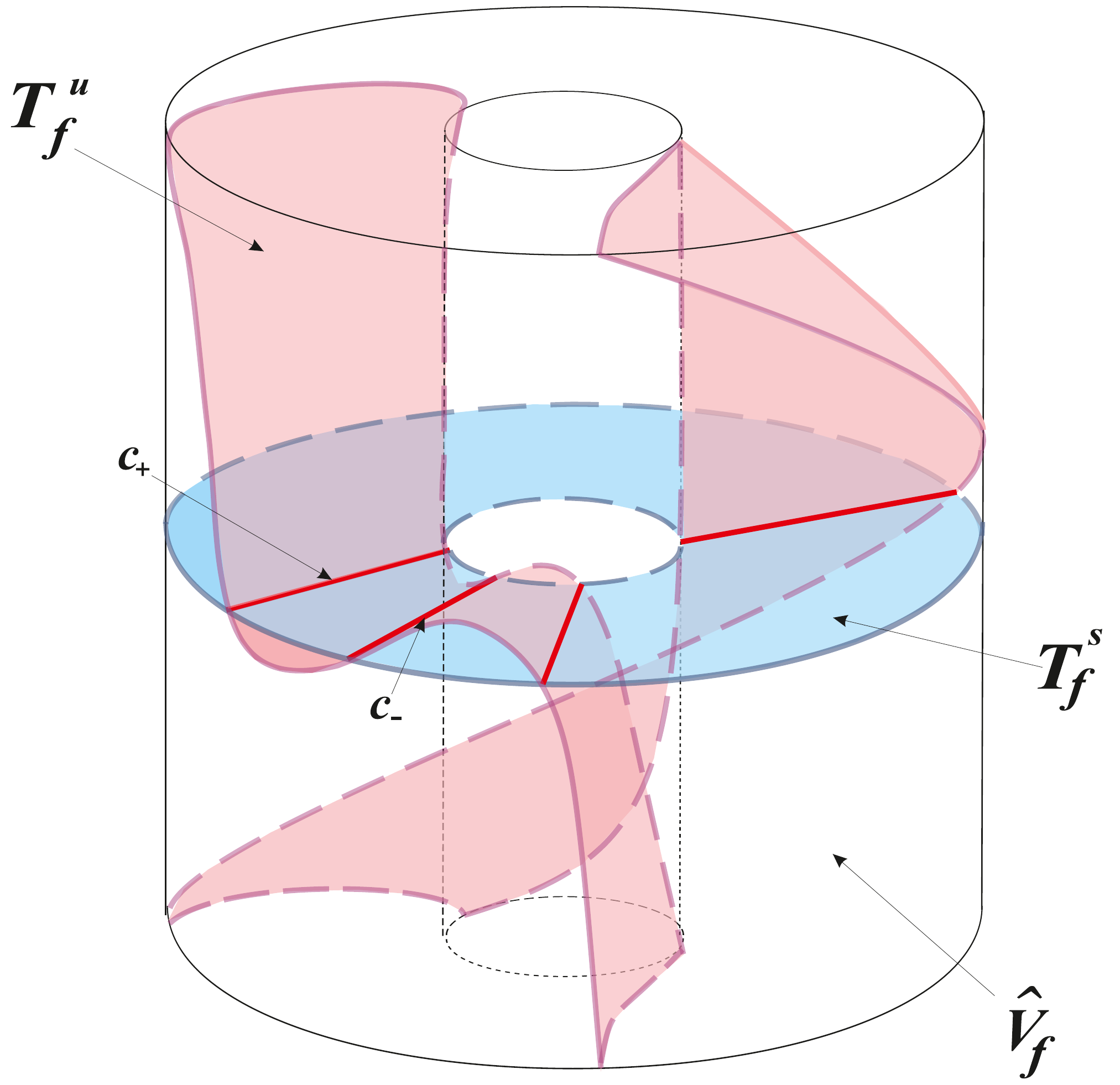}}
\caption{\small Projections of non-compact heteroclinic curves with different orientations into the space $\hat V_f$}
\label{neor-f-p}
\end{figure}  If the set $H_f$ is either empty or orientable, then the lemma is proved. Otherwise, $C_f$ consists of essential pairwise homotopy  on each of the tori $T^s_f,\,T^u_f$ (see, for example, \cite{Ro}) curves and among them there are curves $c_+,c_-$ with the positive, negative orientation, accordingly (see Fig. \ref{neor-f-p}, \ref{nef}).
\begin{figure}[h]
\center{\includegraphics
[width=0.5\linewidth]{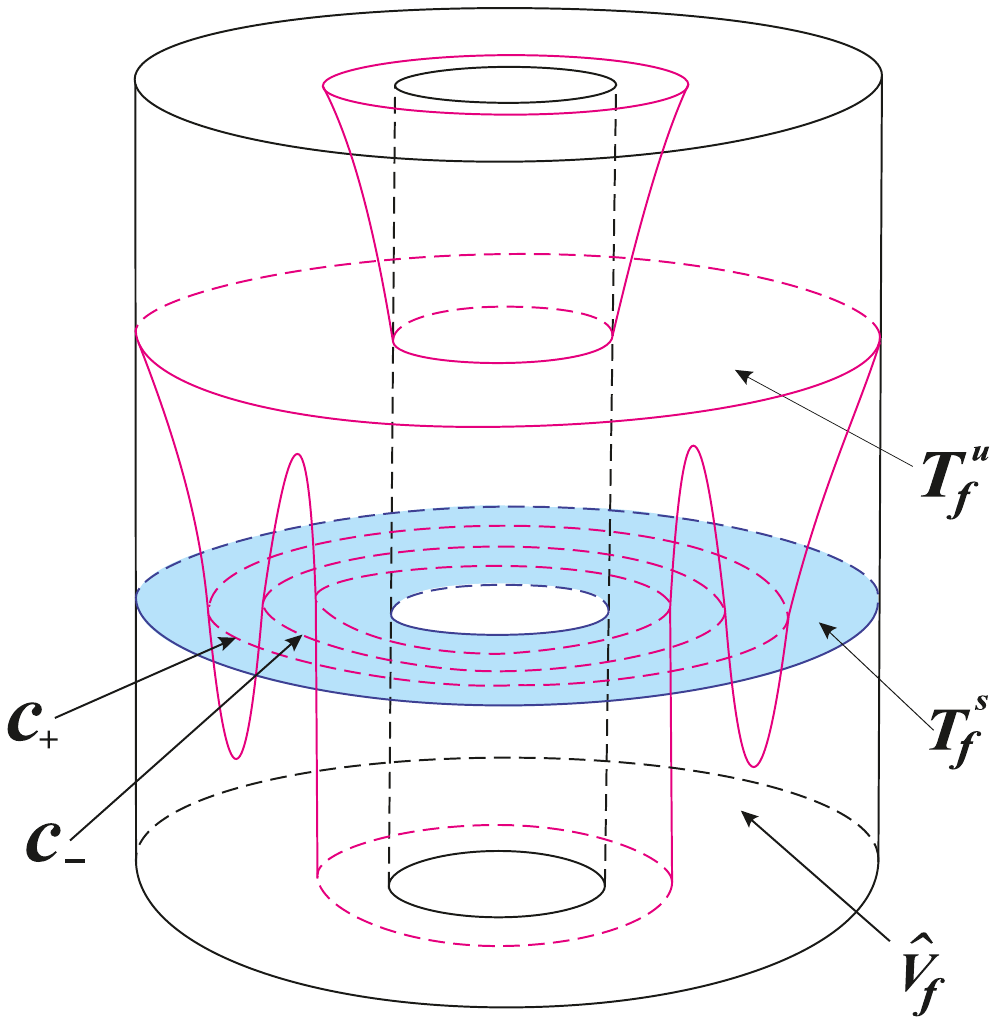}}
\caption{\small Projections of compact heteroclinic curves with  different orientations into the space $\hat V_f$}
\label{nef}
\end{figure}

We show that the number of curves in $H_f$ can be reduced by at least two. 
\begin{figure}[h]
\center{\includegraphics
[width=0.45\linewidth]{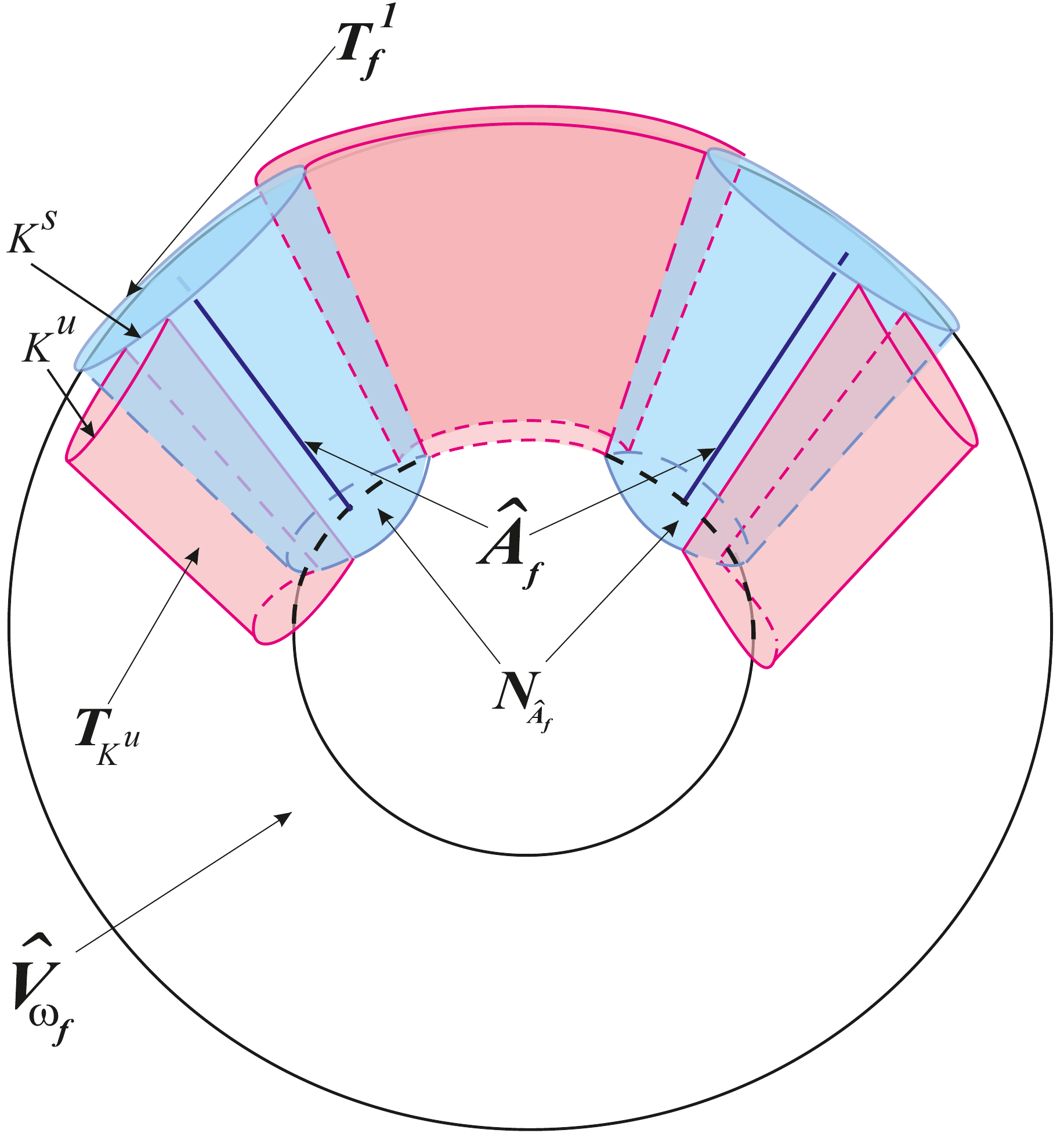}}
\caption{\small Projections of invariant saddle manifolds into the space $\hat V_{\omega_f}$ corresponding to the figure \ref{neor-f-p}}
\label{S2S1}
\end{figure} 

To do this, put $Y_{f}=p_{_{\omega_f}}(W^u_{\sigma^2_f}\cap V_{\omega_f})$ and $\tilde Y_f=Y_f\setminus int\,(N_{L^1_f}\sqcup N_{L^2_f})$. Then $\tilde Y_f$ consists of a finite number of annuli whose boundaries lie on the tori $T^1_f=\partial N_{L^1_f},\,T^2_f=\partial N_{L^2_f}$. 
\begin{figure}[h]
\center{\includegraphics
[width=0.45\linewidth]{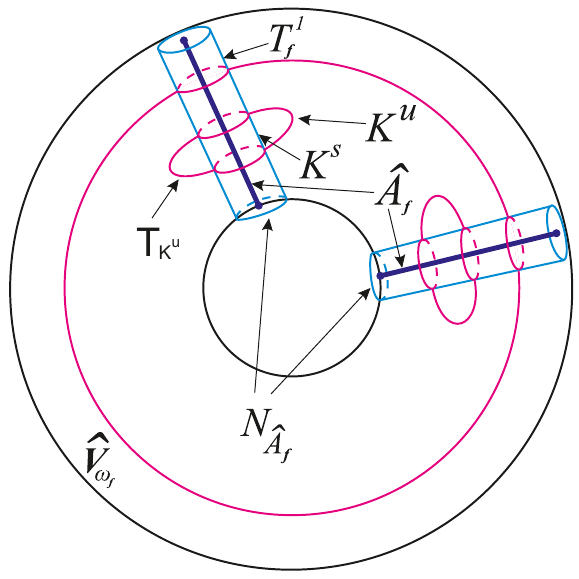}}
\caption{\small Projections of invariant saddle manifolds into the space $\hat V_{\omega_f}$ corresponding to the figure \ref{nef}}
\label{S2S12}
\end{figure}
Due to the non-orientability of the set $H_f$, there is a connected  component $K^u$ of the set $\tilde Y_f$ having boundary circles on the same connected component of the set $T^1_f\sqcup T^2_f$, for certainty we assume that on  $T^1_f$. Then the circles $\partial K^u$ divide the torus $T^1_f$ into two annuli, each of which $K^s$ forms a two-dimensional torus $T_{K^u}$ when combined with the annulus $K^u$. We show that $K^s$ can be chosen such that the torus $T_{K^u}$ is a boundary of a solid torus $Q_{K^u}$, whose the interior avoids $N_{L^1_f}\sqcup N_{L^2_f}$ in $\hat V_{\omega_f}$ (see Fig. \ref{S2S1}, \ref{S2S12}).

Since the torus $T^1_f$ is $\eta_{\omega_f}$-essential in $\hat V_{\omega_f}$, then the annulus $K^u$ can be chosen so that the torus $T_{K^u}$ is also $\eta_{\omega_f}$-essential in $\hat V_{\omega_f}$. By proposition  \ref{f_1}, the torus $T_{K^u}$ bounds a solid torus $Q_{K^u}$ in $\hat V_{\omega_f}$. If $N_{L^1_f}\subset Q_{K^u}$, then, by construction, $cl(Q_{K^u}\setminus N_{L^1_f})$ is also a solid torus bounded by a torus constructed by the second annulus $K^s$. Therefore, everywhere else we assume that $Q_{K^u}\cap N_{L^1_f}=K^s$.

Thus, every annulus $K^u$ is associated with a torus $T_{K^u}$, bounding the solid torus $Q_{K^u}$ in $\hat V_{\omega_f}$. Since, by proposition \ref{ex_6}, any torus homotopically nontrivially embedded in a solid torus bounds a unique solid torus there, then among all such solid tori $Q_{K^u}$ there exists $Q_{K^u_0}$ whose interior does not intersect with  annuli  $K^u$. Then the interior of the torus $Q_{K^u_0}$ does not intersect with the tori $N_{L^1_f}\sqcup N_{L^2_f}$ and $Q_{K^u_0}\cap Y_f=K^u_0$. Denote by $K^s_0$ the second half of the torus $T_{K^u_0}$. 

Denote by $\mathcal K^u_0$ the connected component of the set $T^u_f\setminus C_f$ such that $p_{_f}(p_{_{\omega_f}}^{-1}(K^u_0))\subset\mathcal K^u_0$ and through $\mathcal K^s_0$ the connected component of the set $T^s_f\setminus C_f$ such that the annulus $p_{_f}(p_{_{\omega_f}}^{-1}(K^s_0)),\,\mathcal K^s_0$ lie in the same connected component $N^s_0$ of the set $N_{T^s_f}\setminus T^u_f$. Denote by $\mathcal K'^s_0$ the connected  component of the set $\partial N_{T^s_f}\cap N^s_0$, different from $p_{_f}(p_{_{\omega_f}}^{-1}(K^s_0))$. Let's put $\mathcal K'^u_0=\mathcal K^u_0\cup (cl(N^s_0)\cap T^u_f)$.

By construction $\partial\mathcal K^s_0=\partial\mathcal K^u_0=c_+\sqcup c_-$, where $c_+,\,c_-\subset C_f$ are non-contractible  curves with the positive, negative orientation, respectively. In addition, the torus $\mathcal T_0=\mathcal K^u_0\cup\mathcal K^s_0$ bounds  in $\hat V_f$a solid torus $\mathcal Q_0$, whose interior does not intersect with the set $T^u_f\cup T^s_f$. Denote by $T'^u_{f}$ a two-dimensional torus obtained by smoothing the torus $(T^u_f\setminus\mathcal K'^u_0)\cup\mathcal K'^s_0$. Since the torus $\mathcal T'_0=\mathcal K'^u_0\cup\mathcal K'^s_0$ bounds  in $\hat V_f$ a solid torus $\mathcal Q'_0$, whose interior  does not intersect with the torus $T^u_f$, then there is an isotopic to identity diffeomorphism $\hat h:\hat V_f\to\hat V_f$ such that $\hat h(T^u_f)=T'^u_{f}$. Then by proposition \ref{IsotBan}  there is an arc ${\zeta}_{t}\subset E_f$ such that ${\zeta}_{0}=f, {\zeta}_{1}=f'$ and $T^u_{f'}=T'^u_{f}$, $T^s_{f'}=T^s_f$.

Thus, the diffeomorphism $f'\in G$ is given on the same manifold $M^3$ as the diffeomorphism $f$, but has two less heteroclinic curves. 
Continuing this process, we will construct an arc connecting the diffeomorphism $f$ with some diffeomorphism $f_+\in G^+_p$. 
\end{proof}

\section{Topology of a 3-manifold admitting diffeomorphisms of the class $G$}
In this section we prove Theorems  \ref{TT1} and \ref{exi}. 

\subsection{Lens space as an ambient manifold for diffeomorphisms of class $G$}\label{sec-lin}

Let us prove that if a manifold $M^3$ admits a diffeomorphism $f\in G_p$ then $M^3$ is homeomorphic to a lens space $L_{p,q}$.

\begin{proof} By theorem \ref{te1}, without lost of  generality, we will assume that $f\in G^+_p$, that is, the set of $H_f$ of heteroclinic curves of the diffeomorphism $f$ is orientable and the heteroclinic index is $p\geqslant 0$. Let 's consider the following cases separately: 1) $p=0$, 2) $p>0$.

1) In the case $p=0$, the set $H_f$ is either empty or consists only of compact curves bounding disks on $W^s_{\sigma_f^1}$ containing the saddle $\sigma_f^1$, and all curves in $H_f$ have the same orientation (see Fig. \ref{hecom}). If the set $H_f$ is empty, then, by proposition  \ref{svaz}, the ambient manifold $M^3$ is homeomorphic to $\mathbb S^2\times\mathbb S^1$ (see Fig. \ref{pic3}).

If the set $H_f$ is not empty, then each connected component  $K^u$ of the set $\tilde Y_f$ is a smooth two-dimensional annulus having one boundary component on the torus $T^1_f$, and the other -- on the torus $T^2_f$ and each of the circles is the meridian of the solid torus $N_{L^1_f},\,N_{L^2_f}$, respectively. Denote by $\delta_1\subset N_{L^1_f},\,\delta_2\subset N_{L^2_f}$ two-dimensional disks bounded by these meridians and having exactly one intersection point with the nodes ${L^1_f},\,{L^2_f}$, respectively. Then the set $S=K^u\cup d_1\cup d_2$ is a two-dimensional sphere cylindrical embedded in the manifold $\hat V_{\omega_f}$.   
Since the sphere $S$ has a single intersection point with each of the knots ${L^1_f},\,{L^2_f}$, then, by proposition \ref{ss2s1}, it is ambiently isotopic to the sphere $\mathbb S^2\times\{s_0\},\,s_0\in\mathbb S^1$. 
Let's choose a sphere $\tilde S$, close to the sphere $S$, so that the intersection of $\tilde S\cap N_{L^i_f},\,i=1,2$ is a two-dimensional disk $\tilde d_i$ having a single intersection point with the knot ${L^i_f}$ and $(\tilde S\cap Y_f)\subset int\,(\tilde d_1\sqcup\tilde d_2)$.   

Then the sphere $\bar S$, which is a connected component of the set $p_{\omega_f}^{-1}(\tilde S)$,  bounds a 3-ball $B\subset W^s_{\omega_f}$ containing $\omega_f$ in its interior. In this case, the intersection of $\bar S\cap W^u_{\sigma^2_f}$ belongs to the disjoint union of two disks $\Delta_1\subset p_{\omega_f}^{-1}(\tilde d_1),\,\Delta_2\subset p_{\omega_f}^{-1}(\tilde d_2)$ (see Fig. \ref{Sd}). 
\begin{figure}[h]
\center{\includegraphics[width=0.6\linewidth]{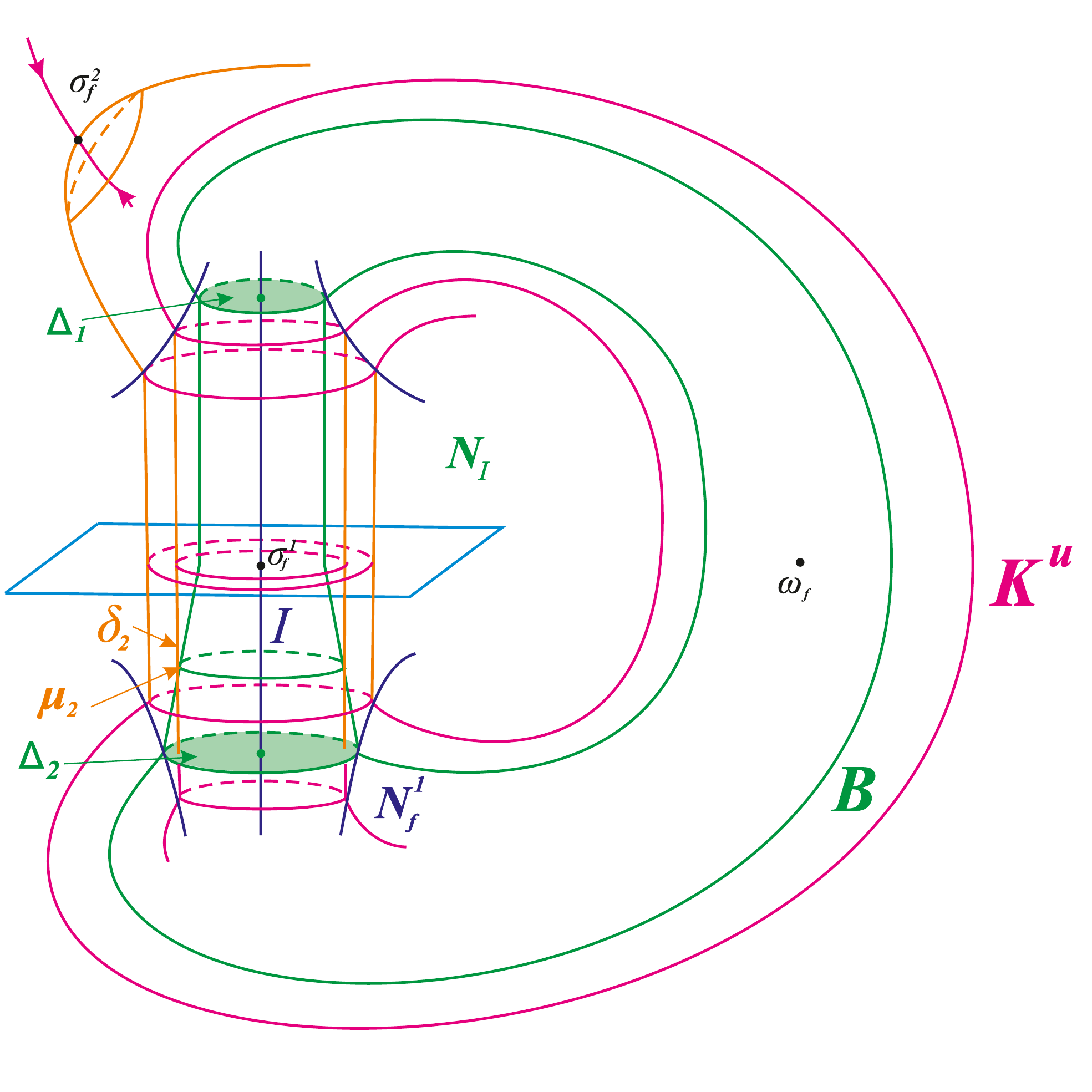}}
\caption{\small Construction of the ball $B$}
\label{Sd}
\end{figure}
Let's put $I=W^u_{\sigma_1}\setminus int\,B$. From the properties of a consistent neighborhoods system and orientability of heteroclinic curves, it follows that there exists a tubular neighborhood $N_I$ of an arc $I$ such that the intersection of $\partial N_I\cap W^u_{\sigma_f^2}$ consists of a single closed curve $\mu_2$. Then the set $Q_1=B\cup N_I$ is homeomorphic to a solid  torus and the curve $\mu_2$ is its meridian.

Since the curve $\mu_2$ is homotopic on $W^u_{\sigma_f^2}\setminus \sigma_f^2$
by the heteroclinic diffeomorphism curve $f$, then it bounds a disk $\delta_2$ containing the saddle $\sigma_f^2$. Let's choose a tubular neighborhood $N_{\delta_2}\subset M^3\setminus int\,Q_1$ of the disk $\delta_2$ so that $N_{\delta_2}\cap W^u_{\sigma_f^2}=\delta_2$ and $N_{\delta_2}\cap\partial Q_1$ is an annulus on the torus $\partial Q_1$, which is a tubular neighborhood of the curve $\mu_2$. Since the curve $\mu_2$ is essential on the torus $\partial Q_1$, the set $S_\alpha=\partial(Q_1\cup N_{\delta_2})$ is homeomorphic to the 2-sphere. By construction, the sphere $S_\alpha$ does not intersect with unstable manifolds of saddle points and, therefore, by proposition \ref{th_2.1.1.}, lies in $W^u_\alpha$, where it bounds a 3-ball $B_\alpha$.

Thus, the set $Q_2=M^3\setminus int\,Q_1$ is so that by cutting it across the disk $\delta_2$, a 3-ball is obtained. This means that $Q_2$ is a solid torus, the curve $\mu_2$ is its meridian, and $M^3=Q_1\cup Q_2$ is the lens space $L_{0,1}\cong\mathbb S^2\times\mathbb S^1$ (see Fig. \ref{hecom}).

2) In the case $p>0$, due to the orientability of the set $H_f$, each connected component $K^u$ of the set $\tilde Y_f$ is an $\eta_{\omega_f}$-essential annulus having boundary circles on different tori. Then each connected component of the boundary of the set $\hat N_f=p_{\omega_f}(N^1_f\cup N^2_f)$  is an $\eta_{\omega_f}$-essential  two-dimensional torus in $\mathbb S^2\times\mathbb S^1$ (see Fig. \ref{N12}). 
\begin{figure}[h]
\center{\includegraphics[width=0.5\linewidth]{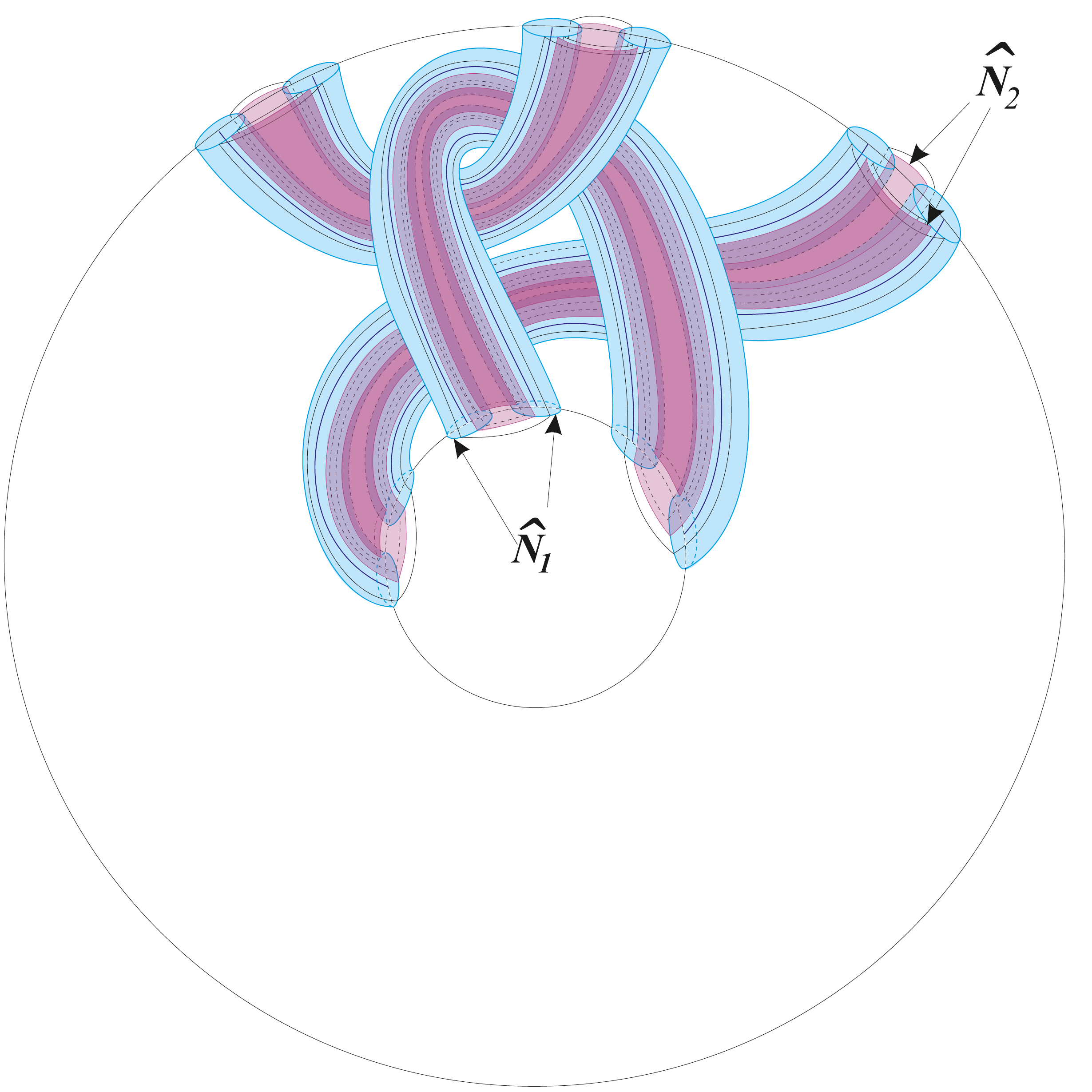}}
\caption{\small Set $\hat N_f$}
\label{N12}
\end{figure}
According to proposition  \ref{f_1}, each such torus bounds in $\hat V_{\omega_f}$ a solid torus, which implies that $\hat N_f$ belongs to the interior of the solid torus $\hat J\subset\hat V_{\omega_f}$. 

Let $J=p_{\omega_f^{-1}}(\hat J)$.    
Since $J$ is an $f$-invariant solid cylinder whose boundary does not intersect the invariant manifolds of saddle points, then, by proposition \ref{th_2.1.1.}, $\partial J\subset W^u_{\alpha_f}$. Let's choose a 2-disk $d\subset (J\cap W^u_{\alpha_f})$ so that $\partial d\subset\partial J$ and $d$ divides $J$ into two connected  components. Select a point $y_0\in int\, d$. Denote by $J_{\omega_f}$ the closure of the connected  component containing $\omega_f$. Then $J_{\omega_f}$ is a 3-ball on the manifold $M^3$, which is tame everywhere, except, perhaps, the point $\omega_f$. Let's put $S_{\omega_f}=\partial J_{\omega_f}$. According to proposition  \ref{sphere-good}  there exists a smooth $3$-ball $B\subset M^3$ such that $\omega_f\in int\, B$ and $\partial B$ transversally intersects $S_{\omega_f}$ along a single curve separating in $S_{\omega_f}$ points ${\omega_f}$ and $y_0$. Without lost of generality, we assume that $\partial B$ intersects the cylinder $J$ along the disk $\Delta$ transversally intersecting $N_f^1$ along two disks and $N_f^2$ -- along $p$ disks (see Fig. \ref{B}).
\begin{figure}[h]
\center{\includegraphics[width=0.5\linewidth]{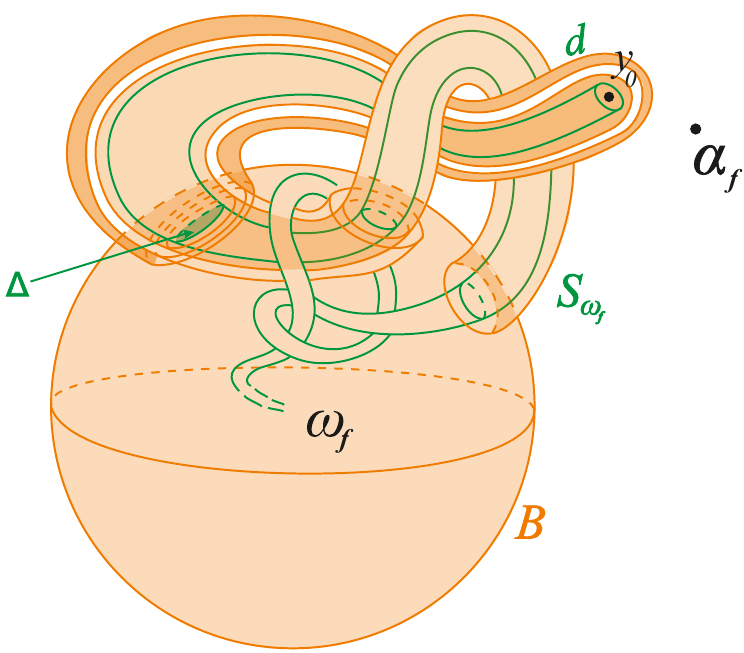}}
\caption{\small Construction of  a ball $B$}
\label{B}
\end{figure}

Let's put $I=W^u_{\sigma_1}\setminus int\,B$. From the properties of the consistent neighborhoods system and orientability of heteroclinic curves, it follows that there exists a tubular neighborhood $N_I$ of an arc $I$ such that $N_I\cap\Delta=N_f^1\cap\Delta$, $W^s_{\sigma_1}$ intersects with $Q_1$ by one 2-disk whose boundary $\mu_1$ intersects with $W^u_{\sigma_2}$ exactly at $p$ points and the intersection of $\partial N_I\cap W^u_{\sigma_f^2}$ consists exactly of $p$ curves. Then the set $Q_1=B\cup N_I$ is homeomorphic to a solid torus and $W^u_{\sigma_2}\cap \partial Q_1=W^u_{\sigma_2}\cap (\Delta\cup\partial N_I)$. Since $cl(W^u_{\sigma_2})\setminus W^u_{\sigma_2}=cl(W^u_{\sigma_1})\subset int\,Q_1$, then $W^u_{\sigma_2}\cap\partial Q_1$ consists of closed curves. Since the intersection of the disk $W^u_{\sigma_2}$ with the torus $\partial Q_1$ is oriented, it consists of a single curve 
$\mu_2$ (see Fig. \ref{ru}). 
\begin{figure}[h]
\center{\includegraphics[width=0.5\linewidth]{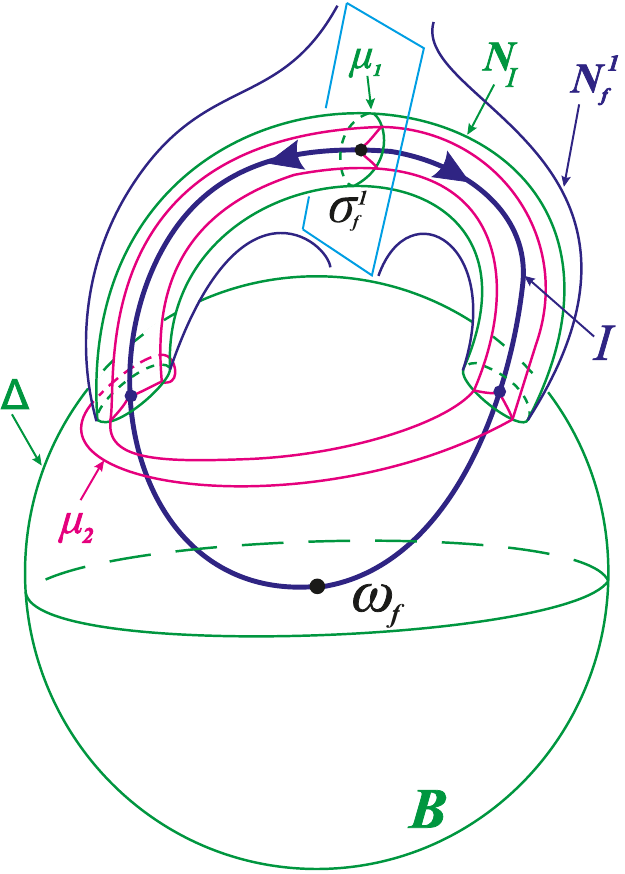}}
\caption{\small Curves $\mu_1,\,\mu_2$}
\label{ru}
\end{figure}

Since the curve $\mu_2$ intersects all heteroclinic curves of the diffeomorphism $f$ in an orientable way, it bounds a disk $\delta_2$ containing the saddle $\sigma_f^2$. Reasoning similarly to the case of $p=0$, we get that $M^3=Q_1\cup Q_2$ is the lens space $L_{p,q}$, where $\langle p,q\rangle$ is the homotopy type of the curve $\mu_2$ on the torus $\partial Q_1$.
\end{proof}

\subsection{Construction of diffeomorphisms with wildly nested separatrices on each lens space}\label{sec-dik}
In this section, we constructively prove theorem \ref{exi}: on any lens space $L_{p,q}$ there exists a diffeomorphism $f\in G$ with wildly embedded one-dimensional saddle separatrices.

\subsubsection{Construction on the lens $L_{0,1}\cong\mathbb S^2\times\mathbb S^1$}
Let $L_1,\,L_2\subset \mathbb S^2\times\mathbb S^1$ be two disjoint knots from generator class (Hopf knots), trivial and non-trivial, respectively. Let $N_{L_1},\,N_{L_2}$ be their pairwise disjoint tubular neighborhoods. Let's choose on the torus $T_i=\partial N_{L_i},\,i=1,2$ generators $\lambda_i,\,\mu_i$ so that the parallel $\lambda_i$ is a Hopf knot, and $\mu_i$ is the meridian
of the solid torus $N_{L_i}$. Let 
$\tilde N_{L_1}\supset N_{L_1}$ be  also a tubular neighborhood of the knot $L_1$ that does not intersect with $N_{L_2}$ and $\tilde T=\partial\tilde N_{L_1}$ (see Fig. \ref{gaga}).
\begin{figure}[h]
\center{\includegraphics[width=0.5\linewidth]{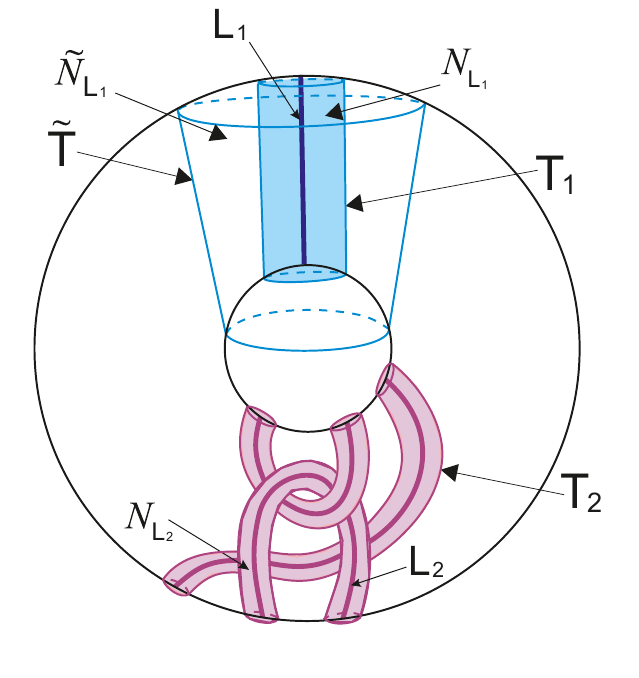}}
\caption{\small Construction of a diffeomorphism $f\in G_0$ with wildly embedded separatrices}
\label{gaga}
\end{figure}

Denote by $\hat V$ a manifold  obtained from 
$\mathbb S^2\times\mathbb S^1\setminus int\,(N_{L_1}\cup N_{L_2})$ by identifying the boundary tori by means of a diffeomorphism that translates the meridian $\mu_1$ into the meridian $\mu_2$. Denote by $q:\mathbb S^2\times\mathbb S^1\setminus int\,(N_{L_1}\cup N_{L_2})\to\hat V$ the natural projection. Let's put $\hat L^s=q(\tilde T)$ and $\hat L^u=q(T_2)$. 
Note that the fundamental group $\pi_1(\hat V)$ admits an  epimorphism $\eta:\pi_1(\hat V)\to\mathbb Z$, which assigns to the homotopy class of a closed curve in $\hat V$ the number of its revolutions around  $q(\lambda_1)$. At the same time, the tori $\hat L^s,\,\hat L^u$ are $\eta$-essential. Let 's put 
$$S=(\hat V,\eta,\hat L^s,\hat L^u).$$ 

By construction, the manifold $\hat V_{\hat L^s}$ is homeomorphic to the initial manifold $\mathbb S^2\times\mathbb S^1$. Since the torus $\tilde T$ bounds two solid tori in $\mathbb S^2\times\mathbb S^1$, then the manifold $\hat V_{\hat L^u}$ is also homeomorphic to the manifold $\mathbb S^2\times\mathbb S^1$. 

Thus, the scheme $S$ is an abstract scheme. By proposition  \ref{rems3},  the scheme $S$ is realizable by some gradient-like diffeomorphism $f\in MS(M^3)$ such that the schemes $S_f$ and $S$ are equivalent. Since the sets ${\hat L^s},\,{\hat L^u},\,\hat V_{\hat L^s},\,\hat V_{\hat L^u}$ are connected, the diffeomorphism $f$ has exactly four non-wandering points of pairwise different Morse indices, that is, $f\in G$. Since the tori ${\hat L^s},\,{\hat L^u}$ do not intersect, the set $H_f$ is empty. According to the theorem \ref{TT1}, the ambient manifold of the diffeomorphism $f$ is homeomorphic to the lens space $L_{0,1}\cong\mathbb S^2\times\mathbb S^1$. According to proposition \ref{wild}, the manifold $W^u_{\sigma^1_f}$ is wildly embedded in the supporting manifold.

\subsubsection{Construction on the lens $L_{p,q},\,p\neq 0$}

Let $p\neq 0$ and $q\neq 0$ be mutually simple with $p$. On the  three-dimensional torus $$\mathbb T^3=\mathbb S^1\times\mathbb S^1\times\mathbb S^1=\left\{\left(e^{i2\pi x},e^{i2\pi y},e^{i2\pi z}\right):\,x,y,z\in\mathbb R\right\}$$ let's set the generators $$a=\mathbb S^1\times\{e^{i2\pi 0}\}\times\{e^{i2\pi 0}\},\,b=\{e^{i2\pi 0}\}\times\mathbb S^1\times\{e^{i2\pi 0}\},\,c=\{e^{i2\pi 0}\}\times\{e^{i2\pi 0}\}\times\mathbb S^1.$$
Let's define two-dimensional tori $\tilde T^s,\,\tilde T^u\subset\mathbb T^3$ as follows: $$\tilde T^s=\left\{\left(e^{i2\pi x},e^{i2\pi y},e^{i2\pi z}\right):\,z=0\right\},\,\tilde T^u=\left\{\left(e^{i2\pi x},e^{i2\pi y},e^{i2\pi z}\right):\,z=\frac {p}{q}y\right\}.$$ Let's choose  tubular neighborhoods of these tori $N_{\tilde T^s},\,N_{\tilde T^u}$. By construction, the closure of each connected component of the set $\mathbb T^3\setminus (N_{\tilde T^s}\cup N_{\tilde T^u})$ is a solid torus with a generator homotopic to the knot $a$. Let's choose one such component $W$ and denote by $\mu_{_W}$ the meridian of the solid torus $W$ (see Fig. \ref{W}). 
\begin{figure}[h]
\center{\includegraphics[width=0.75\linewidth]{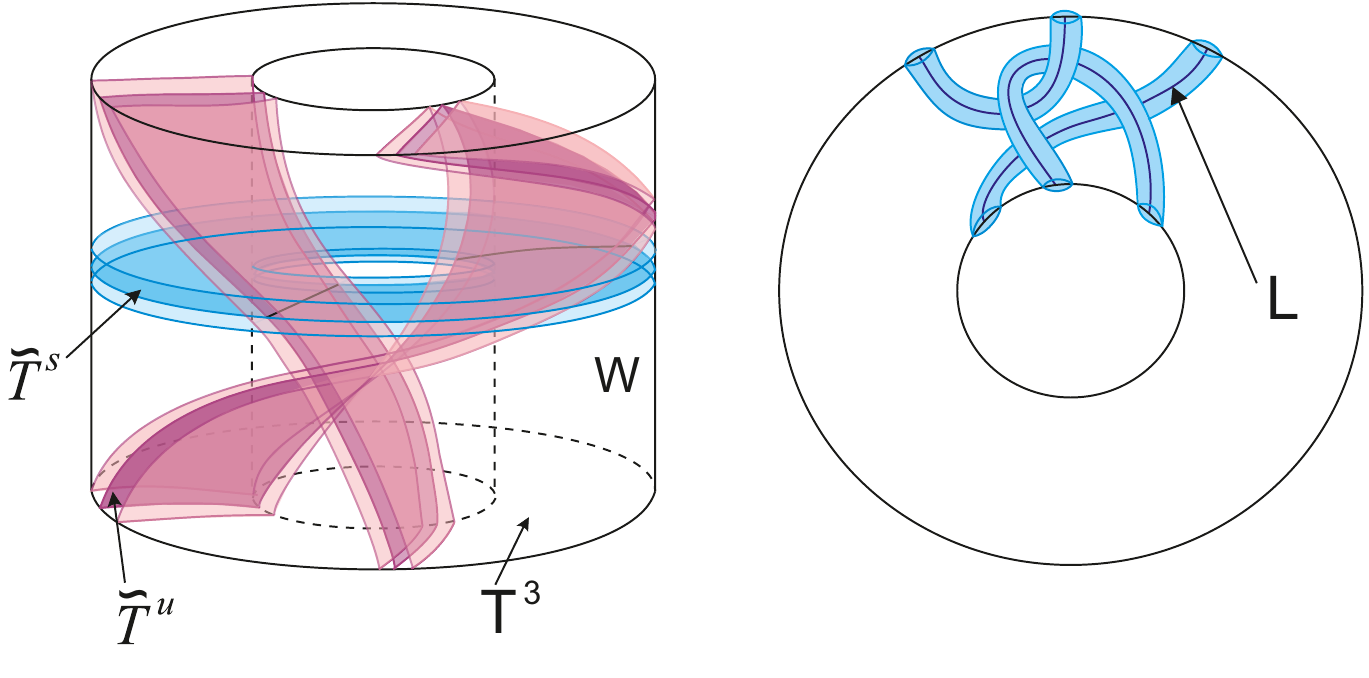}}
\caption{\small Construction of a diffeomorphism $f\in G_2$ with wildly embedded separatrices}
\label{W}
\end{figure}

Let $L\subset\mathbb S^2\times\mathbb S^1$ be a non-trivial Hopf knot, $N_L$ be its tubular neighborhood with meridian 
$\mu_{_{N_L}}$ and $\zeta:\partial N_L\to\partial W$ is a diffeomorphism that translates the meridian $\mu_{_{N_L}}$ into the meridian $\mu_{_W}$. Let's put $$\hat V=(\mathbb T^3\setminus int\,W)\cup_{\zeta}(\mathbb S^2\times\mathbb S^1\setminus int\,N_L).$$ Denote by $q:(\mathbb T^3\setminus int\,W)\sqcup(\mathbb S^2\times\mathbb S^1\setminus int\,N_L)\to\hat V$ the natural projection. Let's put $\hat L^s=q(\tilde T^s),\,\hat L^u=q(\tilde T^u)$. Note that the fundamental group $\pi_1(\hat V)$ admits an  epimorphism $\eta:\pi_1(\hat V)\to\mathbb Z$, which assigns to the homotopy class of a closed curve in $\hat V$ the number of its revolutions around $q(a)$. At the same time, the tori $\hat L^s,\,\hat L^u$ are $\eta$-essential. Let 's put 
$$S=(\hat V,\eta,\hat L^s,\hat L^u).$$

Let's check the validity of the abstract scheme by showing that the manifold $\hat V_{\hat L^s}$ is homeomorphic to the manifold $\mathbb S^2\times\mathbb S^1$ (for the manifold $\hat V_{\hat L^u}$, the proof is similar).

By construction, the manifold $\mathbb T^3\setminus int\,N_{\tilde T^s}$ is homeomorphic to $\mathbb T^2\times[0,1]$. Gluing a solid torus to each component of the connectivity of this manifold so that the meridian of the solid torus is glued to the curve which is homotopic to $b$, we get the manifold $\mathbb S^2\times\mathbb S^1$. In this case, the resulting manifold is a gluing along the boundary of two solid tori $W$ and $\mathbb S^2\times\mathbb S^1\setminus int\,W$. Then the manifold $\hat V_{\hat L^s}$ is obtained by gluing the manifolds 
$\mathbb S^2\times\mathbb S^1\setminus int\,W$, $\mathbb S^2\times\mathbb S^1\setminus int\,N_L$ along the boundary by means  of a diffeomorphism that translates the meridian $\mu_{_{N_L}}$ into meridian $\mu_{_W}$. Since $\mathbb S^2\times\mathbb S^1\setminus int\,W$ is a solid torus, $\hat V_{\hat L^s}$ is homeomorphic to the manifold $\mathbb S^2\times\mathbb S^1$. 

Thus, the scheme $S$ is an abstract scheme. By proposition  \ref{rems3}, the scheme $S$ is realizable by some gradient-like diffeomorphism $f\in MS(M^3)$ such that the schemes $S_f$ and $S$ are equivalent. Since the sets ${\hat L^s},\,{\hat L^u},\,\hat V_{\hat L^s},\,\hat V_{\hat L^u}$ are connected, the diffeomorphism $f$ has exactly four non-wandering points of pairwise different Morse indices, that is, $f\in G$. Since the tori ${\hat L^s},\,{\hat L^u}$ intersect orientably along $p$ $\eta$-essential curves, the set $H_f$ is orientable and consists of $p$ non-compact heteroclinic curves. According to theorem  \ref{TT1}, the ambient  manifold of the diffeomorphism $f$ is homeomorphic to the lens space $L_{p,q}$. According to proposition \ref{wild}, the manifold $W^u_{\sigma^1_f}$ is wildly embedded into the supporting manifold.

\bibliographystyle{ieeetr}\bibliography{biblio}
\end{document}